\documentclass[hidelinks, onefignum,onetabnum]{siamart250211}

\usepackage{lipsum}
\usepackage{amsfonts}
\usepackage{amssymb}
\usepackage{mathrsfs}
\usepackage{graphicx}
\usepackage{epstopdf}
\usepackage{algorithmic}
\usepackage{todonotes}
\usepackage{longtable}
\usepackage{diagbox}
\usepackage{bm}
\usepackage{blkarray}
\usepackage[table]{xcolor}
\usepackage{booktabs}
\usepackage{mwe}
\usepackage{mathtools}
\usepackage{tikz}
\usetikzlibrary{trees,calc,positioning,decorations.markings,arrows,backgrounds,shapes,
 shapes.geometric, fit}
\newcommand{\convexpath}[2]{
[
create hullcoords/.code={
 \global\edef\namelist{#1}
 \foreach [count=\counter] \nodename in \namelist {
   \global\edef\numberofnodes{\counter}
   \coordinate (hullcoord\counter) at (\nodename);
 }
 \coordinate (hullcoord0) at (hullcoord\numberofnodes);
 \pgfmathtruncatemacro\lastnumber{\numberofnodes+1}
 \coordinate (hullcoord\lastnumber) at (hullcoord1);
},
create hullcoords
]
($(hullcoord1)!#2!-90:(hullcoord0)$)
\foreach [
evaluate=\currentnode as \previousnode using \currentnode-1,
evaluate=\currentnode as \nextnode using \currentnode+1
] \currentnode in {1,...,\numberofnodes} {
 let \p1 = ($(hullcoord\currentnode) - (hullcoord\previousnode)$),
 \n1 = {atan2(\y1,\x1) + 90},
 \p2 = ($(hullcoord\nextnode) - (hullcoord\currentnode)$),
 \n2 = {atan2(\y2,\x2) + 90},
 \n{delta} = {Mod(\n2-\n1,360) - 360}
 in
 {arc [start angle=\n1, delta angle=\n{delta}, radius=#2]}
 -- ($(hullcoord\nextnode)!#2!-90:(hullcoord\currentnode)$)
}
}

\definecolor{seabornblue}{rgb}{0.2980392156862745, 0.4470588235294118, 
0.6901960784313725}
\definecolor{seaborngreen}{rgb}{0.3333333333333333, 0.6588235294117647, 
0.40784313725490196}
\definecolor{seabornred}{rgb}{0.7686274509803922, 0.3058823529411765, 
0.3215686274509804}
\definecolor{seabornpurple}{rgb}{0.5058823529411764, 0.4470588235294118, 
0.6980392156862745}
\definecolor{seabornsand}{rgb}{0.8, 0.7254901960784313, 0.4549019607843137}
\definecolor{seaborncyan}{rgb}{0.39215686274509803, 0.7098039215686275, 
0.803921568627451}
\definecolor{seabornorange}{rgb}{0.958, 0.476, 0.206}

\usepackage{caption}
\usepackage{subcaption}
\ifpdf
\DeclareGraphicsExtensions{.eps,.pdf,.png,.jpg}
\else
\DeclareGraphicsExtensions{.eps}
\fi

\newsiamremark{remark}{Remark}
\newsiamremark{hypothesis}{Hypothesis}
\crefname{hypothesis}{Hypothesis}{Hypotheses}
\newsiamthm{claim}{Claim}
\newsiamremark{fact}{Fact}
\crefname{fact}{Fact}{Facts}

\headers{Preconditioners for Multicomponent Diffusion}{K.~Knook, A.~Baier-Reinio, 
P.~E.~Farrell}

\title{Preconditioners for the Onsager--Stefan--Maxwell equations for 
multicomponent diffusion\thanks{
\funding{KK was supported by a Mathematical Institute Scholarship from the 
University of Oxford. ABR was supported by a Clarendon scholarship from the 
University of Oxford. PEF was supported by 
the Engineering and Physical Sciences Research Council [grant no.~EP/W026163/1],
the Science and Technology Facilities Council [grant no.~UKRI/ST/B000495/1],
the Donatio Universitatis Carolinae Chair ``Mathematical modelling of multicomponent systems'',
the UKRI Digital Research Infrastructure Programme through the Science and Technology Facilities Council's Computational Science Centre for Research Communities (CoSeC),
the Swedish Research Council under grant no.~Z2021-06594 while in residence at Institut Mittag-Leffler in Djursholm, Sweden,
and
the National Science Foundation under grant no.~DMS-1929284 while in residence at the Institute for Computational and Experimental Research in Mathematics in Providence, USA.
For the purpose of open access, the authors have applied a CC BY public copyright 
licence to any author accepted manuscript arising from this submission.}}}

\author{Kars Knook\thanks{Mathematical Institute, University of Oxford, UK
(\email{kars.knook@maths.ox.ac.uk}).}
\and Aaron Baier-Reinio\thanks{Mathematical Institute, University of Oxford, UK 
(\email{aaron.baier-reinio@maths.ox.ac.uk})}
\and Patrick E. Farrell\thanks{Mathematical Institute, University of Oxford, UK 
and Mathematical Institute, Faculty of Mathematics and Physics, Charles 
University, Czechia
 (\email{patrick.farrell@maths.ox.ac.uk}).}}

\usepackage{amsopn}

\ifpdf
\hypersetup{
pdftitle={Preconditioners for the Onsager--Stefan--Maxwell equations for 
multicomponent diffusion},
pdfauthor={K. Knook, A. Baier-Reinio P.E. Farrell}
}
\fi

\newsiamremark{assumption}{Assumption}
\crefname{assumption}{Assumption}{Assumptions}

\begin{document}

\maketitle

\begin{abstract}
The Onsager--Stefan--Maxwell (OSM) equations are an important model of mass 
transport in multicomponent flows with multiple chemical species. They describe 
the coupling of diffusive fluxes between species, accounting for their 
interactions through frictional and thermodynamic driving forces. In this work we 
propose an augmented Lagrangian preconditioner and prove its 
discretization-robustness for a Picard linearization of the stationary OSM 
equations in the isobaric, isothermal, ideal gaseous setting. For the Newton 
linearization we employ the augmented Lagrangian preconditioner as a block 
diagonal smoother inside a monolithic geometric multigrid iteration and combine 
with vertex star Schwarz methods. This strategy is shown to be applicable in a 
wide variety of settings which incorporate cross-diffusion, nonideal mixing, 
thermal, pressure, convective, and electrochemical effects. We demonstrate 
robustness or mild dependence with respect to mesh refinement and polynomial 
degree of the proposed monolithic preconditioning strategy for different types of 
multicomponent flows in several applications: cross-diffusion in the human 
airways, separation of gases under a temperature gradient, nonideal mixing of 
benzene and cyclohexane, and electrolytic transport in a Hull cell undergoing 
electroplating.
\end{abstract}

\begin{keywords}
Multicomponent flows, mixtures, Maxwell--Stefan, augmented Lagrangian 
preconditioners, monolithic multigrid preconditioners.
\end{keywords}

\begin{MSCcodes}
65N30, 76T30, 65N55. 
\end{MSCcodes}

\section{Introduction}
\label{sec:introduction}

Many scientific and industrial applications involve multicomponent fluids, those 
consisting of several distinct chemical species.
For example, air consists of nitrogen, oxygen, water vapour, carbon dioxide, and 
other gases. For some applications tracking the individual species is crucial to 
understanding the physical system, such as in the exchange of oxygen and carbon 
dioxide in the lungs. In this work we restrict our attention to the case where all 
species are in a common thermodynamic phase (i.e.~a single liquid or gas).

Many studies of multicomponent flows consider the dilute regime where one species, 
the solvent, is present in much greater quantities than the other species, the 
solutes. In this setting the interactions between the solutes can be neglected. 
The solutes decouple from each other and their diffusive fluxes can be computed 
using Fick's law \cite{fick1855ueber}, yielding decoupled advection-diffusion 
equations. The bulk flow of the mixture can be approximated by the flow of the 
solvent and can be modeled using, for example, the Stokes or Navier--Stokes 
equations. The dilute regime has been studied in great detail 
\cite{bird2002transport, cussler2009diffusion}.
By contrast, in the concentrated regime, where species are present in similar 
quantities, Fick's law is not appropriate as it only accounts for solvent-solute 
interactions. For example, Fick's law neglects cross-diffusion, where the 
diffusion of one species is driven by the concentration gradient of another 
\cite{krishna2019diffusing, krishna1997maxwell, wesselingh2000mass}. In this 
concentrated regime, the diffusive fluxes among species are coupled and can 
instead be modeled by the Onsager--Stefan--Maxwell (OSM) equations 
\cite{bird2002transport, krishna2019diffusing, krishna1997maxwell,
newman2021electrochemical, wesselingh2000mass}.

The OSM equations have applications in many fields including the chemical industry 
\cite{krishna1997maxwell}, combustion \cite{giovangigli2012multicomponent}, 
lithium-ion batteries \cite{richardson2022charge}, and food processing 
\cite{krishna1997maxwell}.
There are many extensions and variations of the OSM equations in the literature,
which can handle applications in, for example, porous media flows 
\cite{veldsink1995use}, multiphase problems \cite{huo2023existence}, and
fluctuating hydrodynamics \cite{balakrishnan2014fluctuating}.
In the present work, we consider variations of the OSM equations that model 
stationary, thermodynamically nonideal mixtures in a common thermodynamic phase.
Our developments herein are also applicable to problems in thermodiffusion
\cite{van2022consolidated} and electrochemistry \cite{van2023structural}.

Finite element discretizations for the OSM equations applied to isobaric, 
isothermal, ideal gaseous mixtures have been discussed in several works 
\cite{braukhoff2022entropy, mcleod2014mixed, van2022augmented}. We consider the 
high-order discretization recently proposed by a subset of the authors 
\cite{baierreinio2024highorderfiniteelementmethods} which models anisobaric 
multicomponent convection-diffusion by coupling the momentum balance of the Stokes 
equation to the OSM equations. This formulation solves for chemical potentials and 
mass fluxes, instead of species concentrations and velocities. Importantly, this 
discretization also extends to nonideal mixtures and condensed phases. Continuous 
well-posedness, discrete well-posedness and quasi-optimality are proven for a 
Picard linearization of the coupled Stokes and OSM equations in 
\cite{aznaran2024finite,baierreinio2024highorderfiniteelementmethods}.

Simulating multicomponent flows is inherently expensive. Each species is described 
by one or more fields (in our case, its mass flux and chemical potential), and 
there are additional fields that characterise bulk mass transport such as pressure 
and barycentric velocity. The dimension of the discretization needed to 
approximate all of these functions quickly exceeds the regime where sparse direct 
solvers for the linearized system are affordable, even for modest numbers of 
species and mesh sizes. Robust and efficient iterative solvers are therefore 
necessary for large-scale simulations.

In this work we propose an augmented Lagrangian (AL) preconditioner to solve a 
symmetric Picard linearization of the OSM equations for which continuous and 
discrete well-posedness and quasi-optimality can be proven in the isobaric, 
isothermal, ideal gaseous setting. In AL preconditioning 
\cite{benzi2006augmented}, the weak formulation is modified with a penalty term 
that gives more control over the Schur complement at the expense of a more 
ill-conditioned top-left block. As a result of the augmentation the Schur 
complement is well approximated by a block diagonal mass matrix. This 
approximation allows us to decouple the mass fluxes from the chemical potentials. 
Furthermore, we can approximate the top-left block with its block diagonal part 
obtaining a block diagonal preconditioner that fully decouples across fields and 
chemical species. We prove and numerically demonstrate discretization-robustness 
of this preconditioner.

However, in practice we use the Newton linearization because of its superior 
convergence and robustness. Since the Picard and Newton linearization share much 
of the same structure, we can use the AL preconditioner for the Newton 
linearization too. To obtain a scalable solution, the diagonal blocks in the 
augmented top-left block are solved effectively using geometric multigrid (GMG) 
with vertex or edge star iteration \cite{brubeck2024multigrid, Farrell_2021}. 
Additive Schwarz methods (ASM) \cite{smith1997domain, toselli2006domain} such as 
vertex/edge star iterations are a key ingredient to obtain robustness and 
parallelisability of preconditioners proposed in this work. A major benefit of the 
AL preconditioner leading to large computational savings is that it completely 
decouples across all fields and species, and thus smaller patches in the 
vertex/edge star iteration are required than if ASM are applied to all fields 
together.

Furthermore, to account for the nonsymmetric structure of the Newton 
linearization, we employ the scalable AL preconditioner as a block diagonal 
smoother for a monolithic GMG method 
\cite{bramble2019multigrid, trottenberg2000multigrid} instead. This monolithic 
approach with block diagonal smoother is more robust and allows us to easily 
include other variables of interest specific to the application, e.g.~mole 
fractions, pressure, and temperature, by scalably and robustly smoothing their 
corresponding diagonal blocks of the Jacobian.

To the best of our knowledge, this is the first work on preconditioners for mixed 
formulations of the OSM equations. The most closely related work is that of 
Giovangigli \cite{giovangigli1991convergent}, which considered a formulation where 
the mass fluxes are eliminated.

We first prove and demonstrate the robustness of the AL preconditioner for the 
Picard linearization of a test problem with a manufactured solution. Next, we 
consider the Newton linearization of the manufactured test problem.
In this setting we compare the performance of the monolithic GMG preconditioner 
with AL smoother to several alternatives, 
including the direct use of the AL preconditioner, and a monolithic GMG 
preconditioner with monolithic smoother.

We then demonstrate the discretization-robustness or mild dependence of our 
proposed preconditioner for the Newton linearization for a wide variety of 
examples encompassing nonideal mixing, thermal, pressure, convective and 
electrochemical effects. Specifically, we consider four test problems with varying 
physical phenomena. First, we consider cross-diffusion in human airways. This is a 
classical three-dimensional multicomponent diffusion problem of nitrogen, oxygen, 
water vapour, and carbon dioxide. Second, we consider a gas separation chamber, 
containing 
a ternary mixture of helium, argon, and krypton that separates under the influence 
of a heat flux. Here temperature is treated as a pseudo-species, pressure is an 
unknown constant, and integral constraints are required to make the problem with 
full Neumann boundary conditions well-posed. Third, we consider the nonideal 
mixing of benzene and cyclohexane, an industrial process that occurs e.g.~in the 
production of nylon. In our formulation for nonideal mixtures, both chemical 
potentials and mole fractions must be discretized. The bulk momentum transport and 
pressure are precomputed using the Stokes equations before solving the OSM 
equations. Finally, we consider electrolytic transport in a Hull cell. Here local 
electroneutrality places an algebraic constraint on the species concentrations, 
which we address using a structure-preserving change of basis, namely the 
salt-charge transformation of Van-Brunt et al.~\cite{van2023structural}.

This remainder of this work is organised as follows. The OSM equations are 
presented in \Cref{sec:OSM_equations}, and the AL preconditioner for the Picard 
linearization of the OSM equations restricted to isobaric, isothermal, ideal 
gaseous mixtures is presented in \Cref{sec:picard}. In \Cref{sec:newton} the 
monolithic GMG preconditioner with AL smoother for the Newton linearization of 
isobaric, isothermal, ideal gaseous mixtures is considered. This preconditioning 
strategy is then extended to anisothermal, nonideal, and electrolytic mixtures in 
Sections \ref{sec:anisothermal}, \ref{sec:nonideal}, and \ref{sec:electrolytic} 
respectively. We offer some conclusions in \Cref{sec:conclusions}.

\section{OSM equations}
\label{sec:OSM_equations}
In this section we present the OSM equations as formulated in 
\cite{baierreinio2024highorderfiniteelementmethods}.
The general form of the OSM equations for $n \geq 2$ species on a bounded 
Lipschitz domain $\Omega \subset \mathbb{R}^d$ with $d \in \{ 2,3 \}$ is
\begin{align}
d_i = \sum_{j=1, j \neq i}^n \frac{RTc_ic_j}{\mathscr{D}_{ij}c_T}\left( v_i - v_j 
\right) \qquad \forall i \in \{1, \dots, n \}.
\label{eq:OSM}
\end{align}
Here, $c_i: \Omega \rightarrow \mathbb{R}$ is the concentration of species $i$ in 
mol m\textsuperscript{$-d$}, $c_T:= \sum_{i=1}^n c_i$ is the total concentration, 
$d_i: \Omega \rightarrow \mathbb{R}^d$ is the diffusion driving force of species 
$i$ in J m\textsuperscript{$-(d+1)$}, $\mathscr{D}_{ij}$ is the Stefan--Maxwell 
diffusivity of species $i$ with respect to species $j$ in
m\textsuperscript{$2$} s\textsuperscript{$-1$}
(only defined for $i \neq j$, and symmetric, so 
$\mathscr{D}_{ij}=\mathscr{D}_{ji}$), $R$ is the ideal gas constant in J 
K\textsuperscript{$-1$} mol\textsuperscript{$-1$}, $T: \Omega \rightarrow 
\mathbb{R}$ is the absolute temperature in K, and $v_i: \Omega \rightarrow 
\mathbb{R}^d$ is the velocity of species $i$ in m s\textsuperscript{$-1$}. The OSM 
equations can be written in shorthand as
\begin{align}
 d_i = \sum_{j=1}^n \bm{M}_{ij}v_j \qquad \forall i \in \{1, \dots, n \},
 \label{eq:OSM_simplified}
\end{align}
where $\bm{M}$ is Onsager's transport matrix, given by
\begin{align}
\bm{M}_{ij} = \begin{cases}
 - \frac{RTc_ic_j}{\mathscr{D}_{ij}c_T}, & \text{if $i \neq j$,} \\
 \sum_{k=1,k \neq i}^n \frac{RTc_ic_k}{\mathscr{D}_{ik}c_T}, & \text{if $i =j$,} 
 \label{eq:onsager_transport_matrix}
\end{cases}
\end{align}
in J s m\textsuperscript{$-(d+2)$}.

The matrix $\bm{M}$ is symmetric positive semi-definite 
\cite{onsager1945theories, van2022augmented}; its positive semi-definiteness
ensures non-negativity of entropy production. This property is also useful for
proving the continuous and discrete well-posedness of a Picard linearized OSM 
problem
\cite{aznaran2024finite,baierreinio2024highorderfiniteelementmethods,van2022augmented}.
Note that the right-hand side of the OSM equations is invariant under a common 
velocity shift, i.e.~$v_i + v$ for all $i$ still solves \eqref{eq:OSM} for 
arbitrary $v: \Omega \rightarrow \mathbb{R}^d$. Physically, this invariance 
encodes the distinction between diffusion (which generates entropy) and convection 
(which does not). Correspondingly, one readily verifies that $\bm{M}$ has a
one-dimensional nullspace spanned by $[1, \ldots, 1]^{\top}$, provided that
all concentrations are strictly positive, which we assume throughout this work.
Fortunately, the left-hand 
side of \cref{eq:OSM_simplified} necessarily lies in the range of $\bm{M}$
(i.e.~$\sum_{i=1}^n d_i = 0$), owing to the Gibbs--Duhem equation.
To fix the component of 
the solution in the nullspace of $\bm{M}$,
it is typical to specify the bulk velocity 
$v_{\text{bulk}}: \Omega \rightarrow \mathbb{R}^d$. The bulk velocity is related 
to the unknown species velocities by the so-called mass-average constraint
\begin{align}
 v_{\text{bulk}} = \frac{\sum_{i=1}^n M_i c_iv_i}{\rho},
 \label{eq:vbulk}
\end{align}
where $M_i>0$ is the molar mass of species $i$ in kg mol\textsuperscript{$-1$}, 
and $\rho := \sum_{i=1}^n M_ic_i$ is the density in kg m\textsuperscript{$-d$}. In 
the context of the standalone OSM equations, the bulk velocity is treated as given 
fixed data (often assumed to be zero). In formulations where momentum transport
is fully coupled to the OSM problem
\cite{aznaran2024finite,baierreinio2024highorderfiniteelementmethods},
the bulk velocity is solved for with a Cauchy momentum equation. 

Furthermore, the 
OSM equations are coupled to mass continuity equations of the form
\begin{align}
 \frac{\partial c_i}{\partial t} + \nabla \cdot (c_iv_i) = r_i \qquad \forall i 
 \in \{ 1, \dots, n\},
\end{align}
where $r_i: \Omega \rightarrow \mathbb{R}$, the rates of generation or depletion 
of species $i$ in mol m\textsuperscript{$-d$} s\textsuperscript{$-1$}, are given. 
In this work we will consider the steady-state problem, so that the continuity 
equations become
\begin{align}
 \nabla \cdot (c_i v_i) = r_i \qquad \forall i \in \{ 1, \dots, n\}.
 \label{eq:cont}
\end{align}

Combining \eqref{eq:OSM_simplified}, \eqref{eq:vbulk}, and \eqref{eq:cont}, we 
obtain the steady-state OSM equations. We will use bar notation to denote 
$n$-tuples, e.g.~$\bar{r} = (r_1, \dots, r_n)$. For given data $v_{\text{bulk}}$ 
and $\bar{r}$, we consider the problem of finding $\bar{c}$ and 
$\bar{v}$ such that:
\begin{subequations}
 \begin{alignat}{2}
     d_i &= \sum_{j=1}^n \bm{M}_{ij}v_j \qquad &&\forall i \in \{1, \dots, n \}, 
     \label{eq:ideal_OSMa} \\
     \nabla \cdot (c_i v_i) &= r_i \qquad &&\forall i \in \{ 1, \dots, 
     n\},\label{eq:ideal_OSMb} \\
     v_{\text{bulk}} &= \frac{\sum_{i=1}^n M_i c_iv_i}{\rho}. \label{eq:ideal_OSMc}
 \end{alignat}
 \end{subequations}
The diffusion driving forces $\bar{d}$ vary from problem to problem.
A typical example for isothermal mixtures is
\begin{align}
 d_i = - c_i\left ( \nabla \mu_i - \frac{M_i}{\rho}\nabla p \right ) 
 \qquad \forall i \in \{ 1, \dots, n\},
 \label{eq:diffusion_driving_forces}
\end{align}
where $\mu_i: \Omega \rightarrow \mathbb{R}$ is the chemical potential of species 
$i$ in J mol\textsuperscript{$-1$} and $p: \Omega \rightarrow \mathbb{R}$ is the 
thermodynamic pressure in J m\textsuperscript{$-d$}.
The chemical potentials are defined by a thermodynamic constitutive law as a 
function of $T$, $p$ and the mole fractions $\bar{\chi}$ which are defined as 
$\chi_i := \frac{c_i}{c_T}$ for each $i$.

Next, following
\cite{aznaran2024finite,baierreinio2024highorderfiniteelementmethods},
we adopt the augmentation strategy from
\cite{van2022augmented}, which was
originally described in \cite{ern1994multicomponent, helfand1960inversion}, to 
obtain
\begin{align}
\label{eq:augmented_OSM}
 d_i + \gamma \omega_i v_{\text{bulk}}= \sum_{j=1}^n \bm{M}_{ij}^\gamma v_j \qquad 
 &\forall i \in \{1, \dots, n \},
\end{align}
by adding \eqref{eq:ideal_OSMc} to \eqref{eq:ideal_OSMa} weighted by $\gamma 
\omega_i$ where $\gamma> 0$ is an augmentation parameter and $\omega_i := 
\frac{M_ic_i}{\rho}$ is the mass fraction of species $i$. The matrices $\bm{M}$ 
and 
$\bm{M}^\gamma$ are related by the identity $\bm{M}^\gamma_{ij} := \bm{M}_{ij} + 
\gamma \omega_i \omega_j$. The augmentation strategy is used to enforce the mass-average constraint. It replaces the symmetric 
positive semi-definite $\bm{M}$ with the uniformly symmetric positive definite 
$\bm{M}^\gamma$, i.e.~there exists a positive lower bound for the smallest 
eigenvalue of $\bm{M}^\gamma$ that holds at all points in $\Omega$,
assuming the species concentrations are bounded away from zero.
This is used
to prove continuous and discrete well-posedness and quasi-optimality of the 
Picard linearization in \Cref{sec:picard}. In practice, we will choose $\gamma =1$ 
for nondimensionalized problems as that has shown the smallest solver times 
\cite{van2022augmented}, and corresponds to enforcing the OSM equation and the 
mass-average constraint equally.

Lastly, we write the OSM equations in terms of the mass fluxes $J_i: \Omega 
\rightarrow \mathbb{R}^d$ instead of the velocities $v_i$ for each $i$. If the 
mass fluxes are discretized instead of the velocities, divergence-conforming 
symmetric stress elements for the Newtonian stress tensor are not needed in the 
full coupling of the OSM equations with the Navier--Stokes equations 
\cite{baierreinio2024highorderfiniteelementmethods}. The mass fluxes $\bar{J}$ are 
given by
\begin{align}
 J_i := M_ic_iv_i \qquad &\forall i \in \{1, \dots, n \}.
\end{align}
For symmetry of the forthcoming linearized problem we shall divide 
\eqref{eq:augmented_OSM} by $M_ic_i$, which yields
\begin{subequations} \label{eq:augmented_OSM__}
\begin{alignat}{2}
 \frac{d_i}{M_ic_i} + \frac{\gamma}{\rho} v_{\text{bulk}} &= \sum_{j=1}^n 
 \widetilde{\bm{M}}_{ij}^\gamma J_j \qquad &&\forall i \in \{1, \dots, n \}, 
 \label{eq:augmented_OSMa_} \\
 \frac{1}{M_i} \nabla \cdot J_i &= r_i \qquad &&\forall i \in \{ 1, \dots, n\}, 
 \label{eq:augmented_OSMb_}
\end{alignat}
\end{subequations}
where $\widetilde{\bm{M}}^\gamma_{ij}:=\frac{\bm{M}_{ij}^\gamma}{M_iM_jc_ic_j}$.

\section{Picard linearization for isobaric, isothermal, ideal gaseous mixtures}
\label{sec:picard}

In an isobaric and isothermal mixture the diffusion driving forces simplify to
\begin{align}
 d_i = -c_i \nabla \mu_i \qquad \forall i \in \{ 1, \dots, n\}.
\end{align}
If we further assume an ideal gaseous mixture,
the Stefan--Maxwell diffusivities $\mathscr{D}_{ij}$ for $i \neq j$
are strictly positive constants.
Ideal gaseous mixtures also satisfy the
following thermodynamic constitutive law:
\begin{align}
\label{eq:ideal_constitutive_law}
 \mu_i = RT \log \left ( c_i \frac{RT}{p^{\text{ref}}}\right ) + 
 \mu_i^{\text{ref}} \qquad \forall i \in \{ 1, \dots, n\},
\end{align}
where $p^{\text{ref}}$, the reference pressure, and $\mu_i^{\text{ref}}$, the 
reference chemical potential of species $i$, are given constants.
Since we have an explicit 
relation between $\bar{\mu}$ and $\bar{c}$, we choose to solve for $\bar{\mu}$ 
instead. Note that this construction gives the added benefit that $c_i > 0$ for 
each $i$, if $c_i$ is computed using $\mu_i$. The problem then becomes: for given 
$v_{\text{bulk}}$, $\gamma>0$, and 
$\bar{r}$, find $\bar{\mu}$ and $\bar{J}$ such that
\begin{subequations} \label{eq:augmented_OSM_}
\begin{alignat}{2}
 -\frac{1}{M_i}\nabla \mu_i + \frac{\gamma}{\rho} v_{\text{bulk}} &= \sum_{j=1}^n 
 \widetilde{\bm{M}}_{ij}^\gamma J_j \qquad &&\forall i \in \{1, \dots, n \}, 
 \label{eq:augmented_OSMa} \\
 \frac{1}{M_i} \nabla \cdot J_i &= r_i \qquad &&\forall i \in \{ 1, \dots, n\}, 
 \label{eq:augmented_OSMb}
\end{alignat}
\end{subequations}
where $\bar{c}$ is given by inverting \eqref{eq:ideal_constitutive_law}.

The system is closed with the boundary conditions
\begin{subequations} \label{eq:bcs}
 \begin{alignat}{3} 
J_{i}\cdot \hat{n} &= g_{i} \in L^2(\Gamma_{N_i}) \qquad && \text{on} \; 
\Gamma_{N_i} \qquad &\forall i \in \{1, \dots, n \}, \label{eq:bc_J} \\
\mu_i &= f_{i} \in H^{\frac{1}{2}}(\Gamma_{D_i}) \qquad && \text{on} \; 
\Gamma_{D_i} \qquad &\forall i \in \{1, \dots, n \}, \label{eq:bc_mu}
 \end{alignat}
\end{subequations}
where $\Gamma_{N_i}$ and $\Gamma_{D_i}$ partition $\partial \Omega$ for all $i$, 
and $\hat{n}$ denotes the outward unit normal.

In the remainder of this section we study a Picard linearization of 
\eqref{eq:augmented_OSM_} in which $\bar{c}$ (and hence $\rho$ and 
$\widetilde{\bm{M}}_{ij}^\gamma$) are treated as fixed. This makes 
\eqref{eq:augmented_OSM_} linear in the variables $\bar{\mu}$ and $\bar{J}$. We 
obtain a fixed point iteration by alternating solves for $\bar{\mu}$ and 
$\bar{J}$, and direct evaluation of $\bar{c}$, $\rho$ and 
$\widetilde{\bm{M}}_{ij}^\gamma$. More concretely, let $\bar{\mu}^l$ and $\bar{J}^l$ 
be the $l^{\mathrm{th}}$ solution iterate of the fixed point iteration. Then 
$\bar{c}^l$, $\rho^l$ and $\widetilde{\bm{M}}_{ij}^{\gamma,l}$ can be found through 
direct evaluation using $\bar{\mu}^l$. To obtain $\bar{\mu}^{l+1}$ and 
$\bar{J}^{l+1}$, we solve the Picard linearization: for given $v_{\text{bulk}}$, 
$\gamma>0$, $\bar{r}$, $\rho^l$ and $\widetilde{\bm{M}}_{ij}^{\gamma,l}$, find 
$\bar{\mu}^{l+1}$ and $\bar{J}^{l+1}$ such that
\begin{subequations} \label{eq:augmented_OSM_picard}
\begin{alignat}{2}
 -\frac{1}{M_i}\nabla \mu_i^{l+1} + \frac{\gamma}{\rho^l} v_{\text{bulk}} &= 
 \sum_{j=1}^n \widetilde{\bm{M}}_{ij}^{\gamma,l} J_j^{l+1} \qquad &&\forall i \in \{1, 
 \dots, n \}, \label{eq:augmented_OSMa_picard} \\
 \frac{1}{M_i} \nabla \cdot J_i^{l+1} &= r_i \qquad &&\forall i \in \{ 1, \dots, 
 n\}.
 \label{eq:augmented_OSMb_picard}
\end{alignat}
\end{subequations}
We will first derive a weak formulation for this Picard linearization and prove 
its well-posedness, then we will show well-posedness and quasi-optimality for a 
finite element discretization, and finally we will present and numerically 
demonstrate a provably robust preconditioner with respect to mesh size and finite 
element degree.
The arguments and proofs in \Cref{sec:weak_form_picard,sec:fem_picard} are 
simplifications of those given in 
\cite{baierreinio2024highorderfiniteelementmethods}, which consider the equations 
coupled to an additional momentum balance.

It is worth discussing the role of $T$ and $p$ in this isothermal, isobaric, ideal 
gaseous setting.
The situation regarding $T$ is straightforward; $T$ can be any specified positive constant. The situation regarding $p$ is more subtle.
As we are assuming an isobaric mixture,
the pressure $p \in \mathbb{R}$ is a constant.
A na\"ive inspection of the closed system
\eqref{eq:ideal_constitutive_law}--\eqref{eq:bcs}
would (incorrectly) appear to suggest that the value of $p$ is physically 
irrelevant, since $p$ does not explicitly appear anywhere in
\eqref{eq:ideal_constitutive_law}--\eqref{eq:bcs}.
However, thermodynamic consistency requires that the concentrations
satisfy the ideal gas volumetric equation of state
\begin{equation}
	\label{eq:state}
	p = c_T R T.
\end{equation}
In particular, if $\cap_{i=1}^n \Gamma_{D_i} \neq \emptyset$ then one can
compute $p$ on $\cap_{i=1}^n \Gamma_{D_i} \subset \partial \Omega$
using the equation of state \cref{eq:state}
together with the constitutive law \cref{eq:ideal_constitutive_law} and
boundary conditions \cref{eq:bc_mu}. The Dirichlet data $\{ f_i \}_{i=1}^n$ must
then be consistent with the requirement that $p$ is a constant, and in this case
the value of $p$ is evidently a known quantity.
On the other hand, if $\cap_{i=1}^n \Gamma_{D_i} = \emptyset$ then $p$
cannot be determined using the boundary data, and in this case
extra integral constraints are typically necessary to close the system.
An example of this case is discussed in detail in \Cref{sec:anisothermal}.

\subsection{Weak formulation and continuous well-posedness}
\label{sec:weak_form_picard}
A weak formulation of \eqref{eq:augmented_OSM_picard} is derived by multiplying 
\eqref{eq:augmented_OSMa_picard} with test functions $\tau_i \in H(\text{div}; \Omega)$ 
and integrating by parts, and multiplying  \eqref{eq:augmented_OSMb_picard} with 
test functions $w_i \in L^2(\Omega)$. For simplicity, we assume full 
homogeneous Neumann boundary conditions, in which case it is appropriate to 
seek $\mu_i^{l+1} \in L^2_0(\Omega) 
:= \{w \in L^2(\Omega): \int_\Omega w \; \mathrm{d}x=0 \}$ for all $i$.
To state the weak formulation of the Picard linearization, we define
\begin{subequations}
\begin{align}
\label{eq:saddlepoint_a}
 a^l(\bar{J}, \bar{\tau}) &:= \sum_{i,j=1}^n \left ( 
 \widetilde{\bm{M}}^{\gamma,l}_{ij}J_j, \tau_i \right )_\Omega , \\
\label{eq:saddlepoint_b}
 b(\bar{J}, \bar{w}) &:= -\sum_{i=1}^n \left ( \frac{1}{M_i}\nabla \cdot J_i, w_i 
 \right )_\Omega ,
\end{align}
\end{subequations}
where $\left (\cdot, \cdot \right )_\Omega$ denotes the $L^2(\Omega)$-inner 
product. The weak formulation for the Picard linearization then becomes: for given
$\gamma>0$,
$\widetilde{\bm{M}}^{\gamma,l}$,
$v_{\text{bulk}} \in L^2(\Omega)^d$,
$\bar{r} \in L^2(\Omega)^n$ and
$\rho^l \in L^{\infty}(\Omega)$,
find $J_i^{l+1} \in H_0(\text{div}; \Omega) := 
\{\sigma \in H(\text{div}; \Omega) : 
\sigma \cdot \hat{n} |_{\partial \Omega} = 0 \}$ and $\mu_i^{l+1} \in L^2_0(\Omega)$ such 
that
\begin{subequations}
\label{eq:picard}
\begin{alignat}{3}
 a^l(\bar{J}^{l+1}, \bar{\tau}) &+ b(\bar{\tau}, \bar{\mu}^{l+1}) &&= \left 
 (\frac{\gamma}{\rho^l} v_\text{bulk}, \sum_{i=1}^n \tau_i \right )_\Omega, \\
 b(\bar{J}^{l+1}, \bar{w}) & &&= - \sum_{i=1}^n (r_i, w_i)_\Omega,
\end{alignat}
\end{subequations}
for all $\tau_i \in H_0(\text{div}; \Omega)$ and $w_i \in L^2_0(\Omega)$ for each $i\in \{1, \dots, n \}$. 

Note that \eqref{eq:picard} has a symmetric saddle-point structure for the pair 
$(\bar{J}^{l+1}, \bar{\mu}^{l+1})$. Thus continuous well-posedness can be 
proved by showing, in the appropriate norms, the inf-sup stability of $b$, and 
coercivity of $a^l$ on the kernel of $b$ \cite[Theorem 4.2.1]{boffi2013mixed}.

\begin{lemma} \label{lemma:picard_wp}
 If $c_i^l \in L^\infty(\Omega)$ and $c_i^l \geq c_{\mathrm{min}}$ a.e.~in 
 $\Omega$ for each $i$ where $c_\mathrm{min}>0$ is a constant, 
 and $\mathscr{D}_{ij} > 0$ is a constant for each $i \neq j$,
 then
 the Picard linearization in \eqref{eq:picard} is well-posed.
\end{lemma}
\begin{proof}
 The proof follows from simplifying the arguments in \cite[Section 
 2.6]{baierreinio2024highorderfiniteelementmethods}.
\end{proof}

\subsection{Finite element discretization, discrete well-posedness and 
quasi-optimality}
\label{sec:fem_picard}
Let $\mathcal{T}_h$ be a shape-regular triangulation of $\Omega$ with maximum cell 
diameter $h$. If conforming finite element spaces are employed, and if we have 
discrete inf-sup stability and coercivity, then the 
discrete Picard linearization is well-posed and quasi-optimal \cite[Theorem 
5.2.2]{boffi2013mixed}. If we employ a conforming, divergence-free, and inf-sup 
stable finite element pair for $H(\text{div}; \Omega)$ and $L^2(\Omega)$, we automatically obtain 
discrete inf-sup stability and coercivity, so discrete well-posedness and 
quasi-optimality of the Picard linearization follow. Commonly used conforming 
finite elements for the $H(\text{div}; \Omega)$ and $L^2(\Omega)$ spaces are Raviart--Thomas (RT) 
or Brezzi--Douglas--Marini (BDM),
paired with discontinuous Lagrange (DG) elements. These 
finite element spaces are widely implemented for arbitrary degree $k$. We will use 
$\text{RT}_k$ and $\text{DG}_{k-1}$ spaces as these are the smallest spaces guaranteeing order $k$ error convergence rates in the equipped norms.

\subsection{Preconditioning}
\label{sec:preconditioning}
We achieve robust solves with respect to $h$ and $k$ using an augmented Lagrangian 
(AL) preconditioning strategy. By adding suitable terms to the equations in
\eqref{eq:picard}, the AL preconditioning strategy 
gains more control over the Schur complement \cite{benzi2006augmented}, allowing a 
robust decoupling of $\bar{J}^{l+1}$ and $\bar{\mu}^{l+1}$. We consider augmented 
Lagrangian terms of the form
\begin{subequations}
\label{eq:augmented_lagrangian_term}
\begin{align}
\label{eq:augmented_lagrangian_term_c}
 c^l(\bar{J}, \bar{\tau}) &= \sum_{i=1}^n\kappa^l_i \left(\frac{1}{M_i}\nabla 
 \cdot J_i, \frac{1}{M_i}\nabla \cdot \tau_i \right )_\Omega, \\
 d^l(\bar{\tau}) &= \sum_{i=1}^n\kappa^l_i \left(r_i, \frac{1}{M_i}\nabla \cdot 
 \tau_i \right )_\Omega,
\end{align}
\end{subequations}
where $\kappa^l_i > 0$ for each $i$. We add \eqref{eq:augmented_lagrangian_term} 
to \eqref{eq:picard}, yielding
\begin{subequations}
\label{eq:augmented_lagrangian}
\begin{align}
 a^l(\bar{J}^{l+1}, \bar{\tau}) + c^l(\bar{J}^{l+1}, \bar{\tau}) + b(\bar{\tau}, 
 \bar{\mu}^{l+1}) &= \left (\frac{\gamma}{\rho^l} v_\text{bulk}, \sum_{i=1}^n 
 \tau_i \right )_\Omega + d^l(\bar{\tau}), \\
 b(\bar{J}^{l+1}, \bar{w}) &= - \sum_{i=1}^n (r_i, w_i)_\Omega .
\end{align}
\end{subequations}
Since $\frac{1}{M_i}\nabla \cdot J_i^l = r_i$ for all $i$, the augmented 
Lagrangian 
terms cancel when evaluated at the exact solution $\bar{J}^{l+1}$, and hence they 
do not change the solution of the saddle-point problem. At the discrete level this also holds because $r_i$ is only ever tested against functions in $\text{DG}_{k-1}$, so we can replace $r_i$ with its $L^2(\Omega)$-projection
into $\text{DG}_{k-1}$. Now the discrete mass continuity equation is exactly satisfiable, and the solution of the discrete saddle-point problem is retained.

Let $A^l$ and $B$ denote the linear operators corresponding to \eqref{eq:saddlepoint_a} and \eqref{eq:saddlepoint_b} 
respectively. By construction the operator corresponding to 
\eqref{eq:augmented_lagrangian_term_c} is $\mathrm{diag}(\bar{\kappa}^l) 
B^TM^{-1}_pB$ where $\mathrm{diag}(\bar{\kappa}^l)$ is the operator that scales 
the $[B^TM^{-1}_pB]_{ii}$ operator by $\kappa^l_i$ for each $i$ and $M_p$ is the 
linear operator corresponding to $\sum_{i=1}^n (\mu_i, w_i)_\Omega$. Using this 
notation we can write down the linear operator corresponding to the left-hand 
side of \eqref{eq:augmented_lagrangian} as
\begin{equation}
\bm{J}_P = 
 \begin{pmatrix}
     A^l + \mathrm{diag}(\bar{\kappa}^l) B^TM^{-1}_pB &  B^T \\
      B &  \\
 \end{pmatrix}.
 \label{eq:AL_system}
\end{equation}
To obtain an approximation to the Schur complement $S^l$, we observe that \newline 
$\mathrm{diag}(\bar{\kappa}^l) B^TM^{-1}_pB$ will dominate $A^l$ for large 
$\kappa^l_i$ for each $i$. Hence, if we ignore $A^l$ and let $w_i := 
\frac{\kappa^l_i}{M_i}\nabla \cdot \tau_i$ for each $i$, the following Schur 
complement approximation can be derived by eliminating $\bar{\tau}$:
\begin{align}
 \tilde{s}^l(\bar{\mu}^{l+1}, \bar{w}) := -\sum_{i=1}^n \frac{1}{\kappa^l_i} 
 \left (\mu^{l+1}_i, w_i \right )_\Omega.
 \label{eq:approx_schur_complement}
\end{align}
which gives arbitrarily fast convergence as a preconditioner for the Schur complement as $\bar{\kappa}^l_i \rightarrow \infty$ for each $i$.
Let $\tilde{S}^l$ denote the linear operator corresponding to 
\eqref{eq:approx_schur_complement}. This gives rise to the following block 
diagonal preconditioner for $\bm{J}_P$
\begin{align}
 \begin{pmatrix}
     A^l + \mathrm{diag}(\bar{\kappa}^l) B^TM^{-1}_pB &  \\
      & \tilde{S}^l \\
 \end{pmatrix},
 \label{eq:AL_preconditioner}
\end{align}
which results in at most 3 preconditioned GMRES iterations as $\bar{\kappa}^l_i \rightarrow \infty$ for each $i$ \cite{murphy2000note}.

Now note that $\mathrm{diag}(\bar{\kappa}^l) B^TM^{-1}_pB$ and $\tilde{S}^l$ do 
not 
induce any coupling between species. Therefore, we choose to exclude the 
off-diagonal contributions of $A^l$ and obtain the following block diagonal 
preconditioner
\begin{align}
 P = \begin{pmatrix}
     A^l_1 + \kappa^l_1 B_1^TM^{-1}_{p,1}B_1 \\
     & \ddots \\
     & & A^l_n + \kappa^l_n B_n^TM^{-1}_{p,n}B_n \\
     & & & \tilde{S}_1^l \\
     & & & & \ddots \\
     & & & & & \tilde{S}_n^l \\
 \end{pmatrix},
 \label{eq:AL_preconditioner_decoupled}
\end{align}
where $A^l_i + \kappa^l_i B_i^TM^{-1}_{p,i}B_i$ and $\tilde{S}^l_i$ are the 
linear operators corresponding to \newline
$\left (\widetilde{\bm{M}}_{ii}^{\gamma,l}J^{l+1}_{i} , \tau_{i} \right )_\Omega + 
\frac{\kappa^l_i}{M_i^2} \left(\nabla \cdot J_i^{l+1}, \nabla \cdot \tau_i \right 
)_\Omega$ and $-\frac{1}{\kappa^l_i}\left ( \mu_{i}^{l+1}, w_{i} \right )_\Omega$ 
respectively. To prove $P$ is a discretization-robust preconditioner for $\bm{J}_P$, it suffices to
show spectral equivalence of \eqref{eq:saddlepoint_a} with
\begin{equation}
 a_\mathrm{diag}^l(\bar{J}, \bar{\tau}) = \sum_{i=1}^n \left 
 (\widetilde{\bm{M}}_{ii}^{\gamma,l}J_{i} , \tau_{i} \right )_\Omega,
\end{equation}
since the augmented Lagrangian term is positive semidefinite.
This is done in the following lemma.
\begin{lemma}
\label{lemma:spectral_equivalence}
Under the same assumptions as \Cref{lemma:picard_wp},
the bilinear forms $a^l$ and $a^l_\mathrm{diag}$ are spectrally equivalent.
\end{lemma}
\begin{proof}
 Let $\bm{Z} = \mathrm{diag}((1/(M_1c_1), \dots, 1/(M_nc_n)))$, then 
 $\widetilde{\bm{M}}^\gamma = \bm{Z}\bm{M}^\gamma \bm{Z}^T$. Since $\bm{M}^{\gamma}$ 
 and $\bm{Z}$ are uniformly symmetric positive definite matrices \cite[Lemma 
 3.3]{aznaran2024finite}, it follows that $\widetilde{\bm{M}}^{\gamma,l}$ is also a 
 uniformly symmetric positive definite matrix \cite{ostrowski1959quantitative}. By 
 construction $\mathrm{diag}(\widetilde{\bm{M}}^{\gamma, l})$ is also uniformly 
 symmetric positive definite. So there exist global bounds on the eigenvalues of 
 $\widetilde{\bm{M}}^{\gamma, l}$ and $\mathrm{diag}(\widetilde{\bm{M}}^{\gamma, l})$. 
 Denote these global bounds as $0<\lambda_{\mathrm{min}}(\widetilde{\bm{M}}^{\gamma, 
 l})\leq\lambda_{\text{max}}(\widetilde{\bm{M}}^{\gamma, l})$ and 
 $0<\lambda_{\mathrm{min}}(\mathrm{diag}(\widetilde{\bm{M}}^{\gamma, 
 l}))\leq\lambda_{\text{max}}(\mathrm{diag}(\widetilde{\bm{M}}^{\gamma, l}))$ 
 respectively.

 First consider 
 \begin{subequations}
 \label{eq:spectral_equivalence_left}
     \begin{align}
     a^l(\bar{J}, \bar{J}) &= \sum_{i,j=1}^n \left ( 
     \widetilde{\bm{M}}^{\gamma,l}_{ij}J_j, J_i \right )_\Omega = \int_\Omega 
     \bar{J}^T \widetilde{\bm{M}}^{\gamma,l} \bar{J} \geq 
     \lambda_{\mathrm{min}}(\widetilde{\bm{M}}^{\gamma, l}) \int_\Omega \bar{J}^T 
     \bar{J} \\
     &\gtrsim \lambda_{\text{max}}(\mathrm{diag}(\widetilde{\bm{M}}^{\gamma, l})) 
     \int_\Omega \bar{J}^T \bar{J} \\
     &\geq \int_\Omega \bar{J}^T \mathrm{diag}(\widetilde{\bm{M}}^{\gamma, l})\bar{J} 
     = \sum_{i=1}^n \left (\widetilde{\bm{M}}_{ii}^{\gamma,l}J_{i} , J_{i} \right 
     )_\Omega = a^l_\mathrm{diag}(\bar{J}, \bar{J}).
     \end{align}
 \end{subequations}
 Similarly,
 \begin{subequations}
 \label{eq:spectral_equivalence_right}
     \begin{align}
     a^l(\bar{J}, \bar{J}) &= \int_\Omega \bar{J}^T \widetilde{\bm{M}}^{\gamma,l} 
     \bar{J} \leq \lambda_{\text{max}}(\widetilde{\bm{M}}^{\gamma, l}) \int_\Omega 
     \bar{J}^T \bar{J} \\
     &\lesssim \lambda_{\mathrm{min}}(\mathrm{diag}(\widetilde{\bm{M}}^{\gamma, l})) 
     \int_\Omega \bar{J}^T \bar{J} \\
     &\leq \int_\Omega \bar{J}^T \mathrm{diag}(\widetilde{\bm{M}}^{\gamma, l})\bar{J} 
     = a^l_\mathrm{diag}(\bar{J}, \bar{J}).
     \end{align}
 \end{subequations}
 Together \eqref{eq:spectral_equivalence_left} and 
 \eqref{eq:spectral_equivalence_right} give the desired spectral equivalence
\begin{equation}
 a^l_\mathrm{diag}(\bar{J}, \bar{J}) \lesssim a^l(\bar{J}, \bar{J}) \lesssim 
 a^l_\mathrm{diag}(\bar{J}, \bar{J}),
\end{equation}
for all $\bar{J}$, where the bounding constants are
\begin{equation}
\lambda_{\mathrm{min}}(\widetilde{\bm{M}}^{\gamma, 
l})/\lambda_{\text{max}}(\mathrm{diag}(\widetilde{\bm{M}}^{\gamma, l})) \quad 
\text{and} \quad
 \lambda_{\text{max}}(\widetilde{\bm{M}}^{\gamma, 
 l})/\lambda_{\mathrm{min}}(\mathrm{diag}(\widetilde{\bm{M}}^{\gamma, l})),
\end{equation}
respectively.
\end{proof}
Hence, in the discrete setting, if we solve the Picard linearization using GMRES 
\cite{saad1986gmres} preconditioned with $P$, we obtain $h$ and $k$ independent 
GMRES iteration counts. Note that the spectral equivalence is dependent on 
$\bar{c}^l$, and thus we do not in general expect robust GMRES iteration counts 
across fixed point iterations due to variations in $\bar{c}^l$.

\subsection{Manufactured solution test problem}
\label{sec:MMS_Picard}
We now demonstrate the \newline discretization-robustness of the block diagonal 
preconditioner $P$. We will do this for the manufactured solution test problem 
from \cite[Sec.~6.1]{van2022augmented}. This problem is defined on $\Omega = 
(0,1)^2$ with $n=4$ species. Let $k_1, k_2:\Omega \rightarrow \mathbb{R}$ be 
differentiable functions such that there exists $K_i>0$ such that $|k_i| < K_i$ 
for each $i$. Then letting
\begin{subequations}
\begin{align}
 &c_1 = k_1 + K_1, &c_2 &= -k_1 + K_1, \\
 &c_3 = k_2 + K_2, &c_4 &= -k_2 + K_2,
\end{align}
\end{subequations}
and 
$\mathscr{D}_{13}=\mathscr{D}_{31}=\mathscr{D}_{23}=\mathscr{D}_{32}=\mathscr{D}_{14}=\mathscr{D}_{41}=\mathscr{D}_{24}=\mathscr{D}_{42}$
gives the following exact solution for any $v_{\text{bulk}}$ and $M_i$:
\begin{subequations}
\begin{align} \label{eq:MMS}
 &v_{1} = - \frac{c_{T}}{\beta_{1}} \frac{\nabla c_{1}}{c_{1}} + \frac{ \rho 
 v_{\text{bulk}}}{c_{T}}, &v_{2} &= - \frac{c_{T}}{\beta_{1}} \frac{\nabla 
 c_{2}}{c_{2}} + \frac{ \rho v_{\text{bulk}}}{c_{T}}, \\
 &v_{3} = - \frac{c_{T}}{\beta_{2}} \frac{\nabla c_{3}}{c_{3}} + \frac{ \rho 
 v_{\text{bulk}}}{c_{T}}, &v_{4} &= - \frac{c_{T}}{\beta_{2}} \frac{\nabla 
 c_{4}}{c_{4}} + \frac{ \rho v_{\text{bulk}}}{c_{T}}, 
\end{align}
\end{subequations}
where
\begin{equation}
 \beta_{1} = 2K_1\left(\frac{1}{\mathscr{D}_{12}}+ 
 \frac{1}{\mathscr{D}_{14}}\right), \qquad \beta_{2} = 
 2K_2\left(\frac{1}{\mathscr{D}_{34}}+ \frac{1}{\mathscr{D}_{14}} \right).
\end{equation}
In this numerical example we choose
\begin{equation}
 k_1 = \frac{1}{2}\exp \left (8xy(1-x)(1-y) \right ), \qquad k_2 = 
 \frac{1}{2}\sin(\pi x) \sin ( \pi y),
\end{equation}
and $K_1 = K_2 = 1$. Furthermore
\begin{equation}
 \mathscr{D} = \begin{pmatrix}
      & 2 & 1 & 1 \\
     2 &  & 1 & 1 \\
     1 & 1 &  & 3 \\
     1 & 1 & 3 &  \\
 \end{pmatrix}, \qquad \rho v_{\text{bulk}} = \begin{pmatrix}
     0 \\ 1 
 \end{pmatrix},
\end{equation}
and $RT = p^{\text{ref}}=1$, $M_i =1$, and $\mu^{\text{ref}}_i = 0$ for each $i$. 
We enforce full Dirichlet boundary conditions on $\mu_i$ using the exact solution, 
and the reaction rates are $r_i := \nabla \cdot (c_i v_i)$ for each $i$.
This problem can 
straightforwardly be extended to $\Omega = (0,1)^3$ by choosing
\begin{equation}
 k_1 = \frac{1}{2}\exp \left (8xyz(1-x)(1-y)(1-z) \right ), \qquad k_2 = 
 \frac{1}{2}\sin(\pi x) \sin ( \pi y)\sin ( \pi z),
\end{equation}
and $\rho v_{\text{bulk}} = (0, 1, 0)^\top.$

The Picard iteration is the iterative process of solving \eqref{eq:picard} and 
computing $\bar{c}^{l+1}$ using $\bar{\mu}^{l+1}$. We deem the Picard iteration to 
have converged when the norm of
$(\bar{J}^{l+1}-\bar{J}^{l}, \bar{\mu}^{l+1} - \bar{\mu}^{l})$
in $[L^2(\Omega)]^{[(d+1) \times n]}$ is less than 
$10^{-9}$.
We start the Picard iteration with a thermodynamically consistent 
initial guess for the concentrations $c_i^0$ such that it satisfies the volumetric 
equation of state \eqref{eq:state}.
In general, when the discrete concentrations are computed by 
\cref{eq:ideal_constitutive_law}, they will only satisfy \eqref{eq:state}
up to a small discretization error.
Therefore, at each step of the Picard iteration, we normalize the discrete
concentrations such that \eqref{eq:state} is satisfied to machine precision.

Each Picard linearization is solved using GMRES preconditioned with $P$, with 
direct solvers for the inner subproblems on the block diagonal using MUMPS 
\cite{amestoy2001fully}. We choose 
dimensionally consistent $\kappa_i^l := \alpha M_i^2 
L_\text{ref}^2\max_{j}|\langle\widetilde{\bm{M}}_{ij}^{\gamma, l}\rangle_\Omega|$ 
where $\alpha > 0$ is a dimensionless constant,
$L_\text{ref}$ is the characteristic length 
scale of $\Omega$ ($L_\text{ref}=1$ in this example), 
and $\langle\widetilde{\bm{M}}_{ij}^{\gamma, l}\rangle_\Omega$ is the average of 
$\widetilde{\bm{M}}_{ij}^{\gamma,l}$ in $\Omega$.
Heuristically this corresponds to 
$\kappa_i^l$ being proportional to the largest of the 
$\widetilde{\bm{M}}_{ij}^{\gamma,l}$ for $j \in \{ 
1, \dots, n\}$, but since $\kappa_i^l$ is a scalar we average over $\Omega$.
We now 
choose large $\alpha = 10$ in order for the approximate Schur complement to be 
close to the actual Schur complement. The coarse meshes are $2\times 2$ and 
$2\times 2\times 2$ meshes for $\Omega =(0,1)^2$ and $\Omega =(0,1)^3$ 
respectively, and denote by $m$ the number of uniform refinements of these coarse 
meshes.

All numerical experiments in this work employ Firedrake 
\cite{ham2023c, rathgeber2016firedrake} and solver routines from PETSc 
\cite{balay2019petsc, dalcin2011parallel, kirby2018,mitchell2016} on a node with 64 
CPU cores and 2 TB of RAM\footnote{The code is available on 
\href{https://github.com/KarsKnook/OSMPreconditioners}{https://github.com/KarsKnook/OSMPreconditioners} and software versions are archived on Zenodo \cite{zenodo/Zenodo-20260415.0}.}.
\Cref{tab:MMS_its_picard} confirms that the preconditioner $P$ exhibits robustness 
with respect to $h$ and $k$ in two and three dimensions.

\begin{table}[]
 \centering
 \scriptsize{
     \begin{tabular}{r|*{5}{c}}
     \toprule
     & \multicolumn{5}{c}{$d=2$} 
     \\
     $m\setminus k$ 
     & 1 & 2 & 3 & 4 & 5
     \\
     \midrule
     1 & 10.67 (15) & 10.77 (13) & 11.38 (13) & 11.85 (13) & 11.85 (13)  \\
     \rowcolor{gray!25}
     2 & 11.86 (14) & 11.85 (13) & 11.92 (13) & 12.00 (13) & 12.00 (13) \\
     3 & 12.08 (13) & 12.00 (13) & 12.08 (13) & 12.23 (13) & 12.08 (13) \\
     \rowcolor{gray!25}
     4 & 12.08 (13) & 12.23 (13) & 12.31 (13) & 12.15 (13) & 12.08 (13) \\
     5 & 12.23 (13) & 12.15 (13) & 12.08 (13) &  12.08 (13) & 12.08 (13) \\
     \midrule
     & \multicolumn{5}{c}{$d=3$} 
     \\
     $m\setminus k$ 
     & 1 & 2 & 3 & 4 & 5
     \\
     \midrule
     1 & 10.58 (12) & 10.92 (12) & 11.36 (11) & 11.36 (11) & 11.45 (11) \\
     \rowcolor{gray!25}
     2 & 11.42 (12) & 11.45 (11) & 11.64 (11) & 11.55 (11) & 11.55 (11) \\
     3 & 11.58 (12) & 11.55 (11) & 11.64 (11) & 11.73 (11) & 11.64 (11) \\
     \rowcolor{gray!25}
     4 & 11.64 (11) & 11.73 (11) & 11.73 (11) & 11.73 (11) & \\
     \bottomrule
     \end{tabular}
 }
 \caption{Average number of GMRES iterations for the manufactured solution test 
 problem solved using Picard iteration with preconditioner $P$ with ideal inner solvers. The number of 
 Picard iterations is in brackets.}
 \label{tab:MMS_its_picard}
\end{table}

\section{Newton linearization for isobaric, isothermal, ideal gaseous mixtures}
\label{sec:newton}
The Picard iteration converges slowly or not at all for some problems. We 
therefore prefer Newton's method because of its quadratic rate of convergence and 
greater robustness \cite{baierreinio2024highorderfiniteelementmethods}. We derive 
a finite element discretization of \eqref{eq:augmented_OSM_}--\eqref{eq:bcs} 
similarly to the finite element discretization of the Picard linearization: for 
given $v_{\text{bulk}}$, $\gamma>0$, and $\bar{r}$, find $J_i \in \{\sigma \in 
\text{RT}_k : \sigma \cdot \hat{n}|_{\Gamma_{N_i}} = \hat{g}_i \}$ and $\mu_i \in 
\text{DG}_{k-1}$ such that
\begin{subequations} \label{eq:OSM_discretization}
 \begin{align}
  \sum_{i,j=1}^n \Big( \widetilde{\bm{M}}_{ij}^{\gamma}J_{j}, \tau_{i} 
  \Big)_\Omega - 
  b(\bar{\tau}, \bar{\mu}) - \sum_{i=1}^n \left ( 
  \frac{\gamma}{\rho}v_{\text{bulk}}, \tau_{i}\right )_\Omega &= -\sum_{i=1}^n 
  \frac{1}{M_i} (f_{i}, \tau_{i}\cdot \hat{n})_{\Gamma_{D_i}}, \\
 b(\bar{J}, \bar{w}) &= -\sum_{i=1}^n (r_i, w_{i})_\Omega,
 \end{align}
\end{subequations}
for $\tau_i \in \{\sigma \in \text{RT}_k : \sigma \cdot \hat{n}|_{\Gamma_{N_i}} = 
0 \}$ and $w_i \in \text{DG}_{k-1}$ $\forall i \in \{1, \dots, n \}$.
Here $\hat{g}_i$ is a discrete interpolant of the boundary data $g_i$, such that
$\sigma \cdot \hat{n}|_{\Gamma_{N_i}} = \hat{g}_i$ can be satisfied exactly
for some $\sigma \in \text{RT}_k$.

The linear operator in the Newton linearization,
\begin{align}
\label{eq:jacobian}
\bm{J} = \begin{pmatrix}
 A & B^T + E + F \\
 B & \\
\end{pmatrix},
\end{align}
shares some of its structure with the Picard linearization.
$A$ is in essence the same as $A^l$ except it uses the value for $\bar{c}$ at the 
current iterate of Newton's method. The main differences with the Picard 
linearization are the $E$ and $F$ operators corresponding to $\sum_{i,j=1}^n 
\Big(\widetilde{\bm{M}}_{ij}^{\gamma}J_{j}, \tau_{i}\Big)_\Omega$ and 
$\sum_{i=1}^n \left ( 
\frac{\gamma}{\rho}v_{\text{bulk}}, \tau_{i}\right )_\Omega$ terms differentiated 
with respect to $\bar{\mu}$. Due to the presence of the $E$ and $F$ operators the 
Schur complement acquires an additional contribution when employing the AL 
technique as in \eqref{eq:augmented_lagrangian}. We will neglect this contribution 
in the Schur complement approximation for the Newton linearization. Thus we employ 
$P$ for the Newton linearization also, evaluated at the current Newton iterate.

To make the AL preconditioner fully scalable, the diagonal blocks in $P$ should 
not all be inverted exactly. The $\tilde{S}_i$ blocks can be inverted exactly 
because these are $\text{DG}$ mass matrices which are trivial to invert as they 
have no coupling between cells. For large $\kappa_i^l$ the Schur complement 
approximation gets better, but $A^l_i + \kappa^l_i B_i^TM^{-1}_{p,i}B_i$ becomes 
nearly singular, as $\kappa^l_i B_i^TM^{-1}_{p,i}B_i$ is a singular symmetric 
positive semi-definite operator for each $i$ \cite{lee2007, schoberl1999multigrid}. 
Since $A^l_i + \kappa^l_i B_i^TM^{-1}_{p,i}B_i$ corresponds to an $H(\text{div}; \Omega)$ 
Riesz map with a heterogeneous coefficient in front of the mass term, we can solve 
these problems robustly with respect to $h$ and $k$ using geometric multigrid 
(GMG) with vertex or edge star patch smoothers \cite{arnold2000, 
schoberl1999multigrid}. Vertex/edge star patch smoothers are ASM where the patches 
only include the degrees of freedom on the interior of all topological entities 
incidental to the vertex/edge, as depicted in \Cref{fig:star} for an
$\text{RT}_{1}$ vertex star patch. Whilst edge star patches suffice in 3D, we in 
practice use vertex star patches in the numerical simulations as this better 
demonstrates the robustness of the AL preconditioning strategy. The main benefit 
of splitting across fields and chemical species in $P$ is that the patch problems 
are the smallest they can be while retaining robustness. Small patch problems 
have lower assembly and solve costs, which are especially beneficial for high $k$. 

To further improve the robustness of the augmented Lagrangian preconditioner, we 
employ $P$ as a block diagonal smoother inside a monolithic GMG preconditioner. 
Whereas block diagonal preconditioners naturally employ solvers like multigrid for 
single blocks, monolithic multigrid methods retain the coupling between different 
fields across levels, and generally require careful design of relaxation methods 
to treat the coupled problem on each level; in this context $P$ deals with this 
very effectively. Another benefit of this approach is that for additional 
variables beyond $\bar{J}$ and $\bar{\mu}$ (i.e.~variables specific to different 
applications), their corresponding diagonal blocks of the Jacobian can be smoothed 
in a scalable and discretization-robust manner while still being coupled in a 
monolithic GMG cycle. This avoids having to approximate nested Schur complements. 
A complete solver diagram is shown in \Cref{fig:mono}.

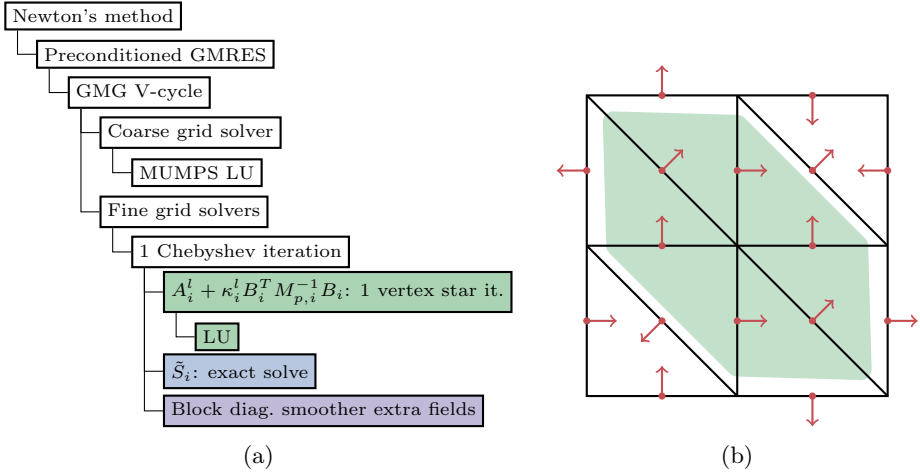
\begin{figure}
  \centering
  \begin{subfigure}[b]{0.56\textwidth}
      \centering
          \scriptsize{
         \begin{tikzpicture}[%
           every node/.style={draw=black, thick, anchor=west},
           bluebox/.style={fill=seabornblue!40},
           greenbox/.style={fill=seaborngreen!50},
           violetbox/.style={fill=seabornpurple!50},
           grow via three points={one child at ([xshift=12pt, 
           yshift=-9pt]\tikzparentnode.south west) and
           two children at ([xshift=12pt, yshift=-9pt]\tikzparentnode.south west) 
           and ([xshift=12pt, yshift=-24pt]\tikzparentnode.south west)},
           edge from parent path={([xshift=5pt]\tikzparentnode.south west) |- 
           (\tikzchildnode.west)}]
         \node {Newton's method}
         child { node {Preconditioned GMRES}
           child { node {GMG V-cycle}
              child{node{Coarse grid solver}
                 child{node{MUMPS LU}}
                 }
              child [missing]
              child{node{Fine grid solvers}
                 child{node{1 Chebyshev iteration}
                     child{node[greenbox]{$A^l_i + \kappa^l_i 
                     B_i^TM^{-1}_{p,i}B_i$: 1 vertex star it.}
                         child{node[greenbox]{LU}}
                         }
                     child [missing]
                     child{ node[bluebox] {$\tilde{S}_i$: exact solve}}
                     child{ node[violetbox] {Block diag.~smoother extra fields}
                         }
                     }
                 }
              }
           };
         \end{tikzpicture}}
         \caption{}
         \label{fig:mono}
  \end{subfigure}
 \begin{subfigure}[b]{0.4\textwidth}
   \centering
       \begin{tikzpicture}[scale=4]
         \foreach \i/\k in {0/0, 0.5/1, 1/2} {
           \foreach \j/\l in {0/0, 0.5/1, 1/2} {
             \coordinate (v\k\l) at (\i, \j);
           }
         }
         \coordinate (o20) at ($(v20) + (135:0.125)$);
         \coordinate (o21) at ($(v21) + (180:0.1)$);
         \coordinate (o12) at ($(v12) + (270:0.1)$);
         \coordinate (o02) at ($(v02) + (315:0.125)$);
         \coordinate (o01) at ($(v01) + (0:0.1)$);
         \coordinate (o10) at ($(v10) + (90:0.1)$);
     
         \path[fill=seaborngreen, opacity=0.35]
         \convexpath{o10,o01,o02,o12,o21,o20}{1pt};
     
         \draw[thick, black, line join=miter] (v00) -- (v10) -- (v20) -- (v21) -- 
         (v22) -- (v12) -- (v02) -- (v01) -- cycle;
         \draw[thick, black, line join=miter] (v01) -- (v11) -- (v21);
         \draw[thick, black, line join=miter] (v10) -- (v11) -- (v12);
         \draw[thick, black, line join=miter] (v01) -- (v10);
         \draw[thick, black, line join=miter] (v02) -- (v11) -- (v20);
         \draw[thick, black, line join=miter] (v12) -- (v21);
     
         \foreach \i in {0.25, 0.75} {
           \foreach \j in {0, 0.5, 1} {
             \draw[fill, seabornred] (\i, \j) circle (0.01);
             \draw[fill, seabornred] (\j, \i) circle (0.01);
           }
         }
         \draw[fill, seabornred] (0.25, 0.25) circle (0.01);
         \draw[fill, seabornred] (0.25, 0.75) circle (0.01);
         \draw[fill, seabornred] (0.75, 0.25) circle (0.01);
         \draw[fill, seabornred] (0.75, 0.75) circle (0.01);
     
         \draw[thick, seabornred, line join=miter, ->] (0,0.25) -- (0.1, 0.25);
         \draw[thick, seabornred, line join=miter, ->] (0,0.75) -- (-0.1, 0.75);
         \draw[thick, seabornred, line join=miter, ->] (0.5,0.25) -- (0.6, 0.25);
         \draw[thick, seabornred, line join=miter, ->] (0.5, 0.75) -- (0.6, 0.75);
         \draw[thick, seabornred, line join=miter, ->] (1,0.25) -- (1.1, 0.25);
         \draw[thick, seabornred, line join=miter, ->] (1,0.75) -- (0.9, 0.75);
         \draw[thick, seabornred, line join=miter, ->] (0.25, 0) -- (0.25, 0.1);
         \draw[thick, seabornred, line join=miter, ->] (0.75, 0) -- (0.75, -0.1);
         \draw[thick, seabornred, line join=miter, ->] (0.25, 0.5) -- (0.25, 0.6);
         \draw[thick, seabornred, line join=miter, ->] (0.75, 0.5) -- (0.75, 0.6);
         \draw[thick, seabornred, line join=miter, ->] (0.25, 1) -- (0.25, 1.1);
         \draw[thick, seabornred, line join=miter, ->] (0.75, 1) -- (0.75, 0.9);
     
         \draw[thick, seabornred, line join=miter, ->] (0.25, 0.25) -- (0.18, 
         0.18);
         \draw[thick, seabornred, line join=miter, ->] (0.25, 0.75) -- (0.32, 
         0.82);
         \draw[thick, seabornred, line join=miter, ->] (0.75, 0.25) -- (0.82, 
         0.32);
         \draw[thick, seabornred, line join=miter, ->] (0.75, 0.75) -- (0.82, 
         0.82);
       
       \end{tikzpicture}
       \caption{}
       \label{fig:star}
     \end{subfigure}
     \caption{Solver diagram for the monolithic GMG preconditioner with AL 
     smoother, and a vertex star patch for the $\text{RT}_1$ discretization. The 
     red arrows correspond to the $\text{RT}_1$ degrees of freedom.}
     \label{fig:monolithic_preconditioner}
\end{figure}

\subsection{Manufactured solution test problem}
We compare three different preconditioners for the manufactured solution test 
problem as in \Cref{sec:MMS_Picard}. The first preconditioner is the AL 
preconditioner $P$ in which the $H(\text{div}; \Omega)$ Riesz maps are solved using GMG 
with vertex star smoothers. The second is the monolithic GMG preconditioner we 
propose with AL smoother, as shown in \Cref{fig:mono}. The third is a fully 
monolithic GMG preconditioner with a vertex Vanka smoother for all variables 
(i.e.~the relaxation iterates over vertices and solves for all degrees of freedom 
in the closure of the star of the vertex). When employing the monolithic GMG 
preconditioner with vertex Vanka smoother no modification using the AL technique 
is necessary. We choose $\alpha =0.1$ as this balances the correctness of the 
Schur complement approximation without making the $H(\text{div}; \Omega)$ Riesz maps too 
ill-conditioned. From now on, we employ the Eisenstat--Walker method 
\cite{eisenstat1996choosing} to adaptively set the tolerance for the linear 
solves, and deem Newton's method converged when the ratio of the Euclidean norm of the nonlinear residual to its initial value is less than $10^{-8}$.

The average GMRES iteration counts for the three preconditioners are presented in 
\Cref{tab:MMS_its}. The monolithic GMG preconditioners both exhibit robust 
iteration counts with respect to $h$ and $k$. We do not observe robustness with 
respect to $h$ for the AL preconditioner alone. We further note that the 
monolithic GMG preconditioner with AL smoother requires at most twice as many 
iterations as the monolithic GMG preconditioner with vertex Vanka smoother, but 
its patches are substantially smaller, and run times much lower, as shown in 
\Cref{tab:patch_sizes}. 

In this numerical experiment the vertex Vanka patches contain approximately four 
to nine times more degrees of freedom ($n$ is the lower limit) than the vertex 
star patches. There are $n$ times as many vertex star patches than vertex Vanka 
patches. Since the patch problems are solved with dense linear algebra, cubic LU
factorization costs are incurred, leading to significantly higher run times.
These are nearly up to two orders of magnitude in some cases, and sometimes the 
comparison cannot be made as the monolithic GMG with vertex Vanka smoother 
required more RAM than what was available. Assembly costs are also higher, with 
more than six times the number of nonzeros. The monolithic multigrid 
preconditioner with AL smoother offers substantially better robustness and 
efficiency than the two alternatives considered here. 

\begin{table}[]
 \centering
 \scriptsize{
     \begin{tabular}{r|*{5}{c}}
     \toprule
     & \multicolumn{5}{c}{$d=2$} 
     \\
     $m \backslash k$
     & 1 & 2 & 3 & 4 & 5
     \\
     \midrule
     1 & 18.6/9.40/4.2 (5) & 18.2/7.4/4.2 (5) & 18.8/6.8/3.4 (5) & 19.6/5.6/3.4 
     (5) & 19.6/5.6/3.4 (5) \\
     \rowcolor{gray!25}
     2 & 25.6/10.0/5.2 (5) & 23.4/7.2/4.0 (5) & 24.4/6.4/3.4 (5) & 25.8/6.0/3.4 
     (5) & 25.8/6.0/3.4 (5) \\
     3 & 34.0/10.2/5.4 (5) & 29.4/6.0/3.4 (5) & 29.2/6.0/3.4 (5) & 29.8/5.8/2.8 
     (5) & 30.2/5.8/3.2 (5) \\
     \rowcolor{gray!25}
     4 & 39.4/10.2/5.6 (5) & 33.0/6.4/3.4 (5) & 34.2/5.6/2.8 (5) & 34.2/4.8/2.6 
     (5) & 34.6/4.8/2.8 (5) \\
     5 & 44.8/10.8/6.2 (5) & 35.6/6.8/3.4 (5) & 36.4/5.4/2.8 (5) & 38.4/5.0/2.6 
     (5) & 38.4/6.2/2.6 (6) \\
     \midrule
     & \multicolumn{5}{c}{$d=3$} 
     \\
     $m\backslash k$ 
     & 1 & 2 & 3 & 4 & 5
     \\
     \midrule
     1 & 19.8/9.30/4.8 (4) &  18.3/6.8/4.3 (4) &  19.5/6.0/3.5 (4) & 20.0/5.5/3.5 
     (4) & 20.3/5.5/3.0 (4) \\
     \rowcolor{gray!25}
     2 & 28.0/10.8/5.8 (4) & 26.3/7.3/4.5 (4) & 26.3/6.0/3.8 (4) & 27.0/5.5/3.5 
     (4) & 27.5/5.3/--- (4) \\
     3 & 38.3/10.3/7.0 (4) & 34.0/7.0/4.5 (4) & 35.8/6.3/--- (4) & 36.3/5.0/--- 
     (4) & 36.8/4.8/--- (4) \\
     \rowcolor{gray!25}
     4 & 45.8/10.5/7.5 (4) & 40.5/6.5/4.3 (4) & 41.0/6.0/--- (4) & ---/4.5/--- (4) 
     & \\
     \bottomrule
     \end{tabular}
 }
 \caption{Average number of GMRES iterations for the manufactured solution test 
 problem solved with AL/monolithic GMG with AL smoother/monolithic GMG with vertex 
 Vanka smoother preconditioners. The number of Newton iterations is given in 
 brackets.}
 \label{tab:MMS_its}
\end{table}

\begin{table}[]
 \centering
 \scriptsize{
     \begin{tabular}{r|*{5}{c}}
     \toprule
     $k$ & 1 & 2 & 3 & 4 & 5 \\
     \midrule
     & \multicolumn{5}{c}{Average patch size} \\
     \midrule
     Vertex star & 33 & 161 & 446 & 953 & 593 \\
     \rowcolor{gray!25}
     Vertex Vanka & 300 & 1240 & 2095 & 4229 & 3703 \\
     Ratio & 9.1 & 7.7 & 4.7 & 4.5 & 6.2 \\
     \midrule
     & \multicolumn{5}{c}{Assembled nonzeros} \\
     \midrule
     Vertex star & $1.10\times10^7$ & $1.63\times10^8$ & $1.51 \times 10^7$ & $5.79 \times 10^7$ & $2.15 \times 10^7$ \\
     \rowcolor{gray!25}
     Vertex Vanka & $7.00 \times 10^7$ & $1.03 \times 10^9$ & $9.59 \times 10^7$ & $3.70 \times 10^8$ & $1.38 \times 10^8$ \\
     Ratio & 6.4 & 6.3 & 6.4 & 6.4 & 6.4 \\
     \midrule
     & \multicolumn{5}{c}{Run time (s)} \\
     \midrule
     Vertex star & $5.05 \times 10^1$ & $3.69 \times 10^2$ & $7.17 \times 10^1$ & $3.10 \times 10^2$ & $1.24 \times 10^3$ \\
     \rowcolor{gray!25}
     Vertex Vanka & $1.67 \times 10^2$ & $6.24 \times 10^3$ & $2.48 \times 10^3$ & $2.42 \times 10^4$ & $1.95 \times 10^4$ \\
     Ratio & 3.3 & 17 & 35 & 78 & 16 \\
     \bottomrule
     \end{tabular}
 }
 \caption{Average patch size and assembled nonzeros on the finest level of GMG, 
 and total run time of the solver for the manufactured solution test problem in 
 three dimensions for each $k$. The comparisons are made for the highest available 
 $m$ for the vertex Vanka smoother.
 }
 \label{tab:patch_sizes}
\end{table}

\subsection{Diffusion in human airways}
\label{sec:lungs}
A physically relevant test problem is the diffusion of nitrogen, oxygen, 
carbon 
dioxide, and water vapour (labeled species 1, 2, 3, and 4 respectively) in the 
human airways \cite[Sec.~6.2]{van2022augmented}. A realistic time-dependent 
lung model must consider convective forces and pressure-driven elastic expansion, 
but in this test problem we only consider the time-averaged diffusion dynamics.

We prescribe the following boundary conditions for the mole fractions at the 
entrance of the trachea and at the end of the bronchi \cite{guyton2006text}:
\begin{align}
\begin{cases}
 \left . \chi_{1} \right |_\text{trac} = 0.7409, & \left . \chi_{1} \right 
 |_\text{bron} = 0.7490, \\
 \left . \chi_{2} \right |_\text{trac} = 0.1967, & \left . \chi_{2} \right 
 |_\text{bron} = 0.1360, \\
 \left . \chi_{3} \right |_\text{trac} = 0.0004, & \left . \chi_{3} \right 
 |_\text{bron} = 0.0530, \\
 \left . \chi_{4} \right |_\text{trac} = 0.0620, & \left . \chi_{4} \right 
 |_\text{bron} = 0.0620,
\end{cases}
\end{align}
and no-flux boundary conditions everywhere else, which correspond to gas 
impermeable walls.
We indirectly implement mole fraction boundary conditions by appropriately
choosing the Dirichlet data $f_i$ in \cref{eq:bc_mu}.
In particular, we first fix an ambient pressure of
$p=101325$ J m\textsuperscript{$-3$} in \cref{eq:state};
the value of $f_i$ can then be calculated unambiguously using
\cref{eq:state} and \cref{eq:ideal_constitutive_law}.

We set $v_{\text{bulk}} = 0$ and $r_i = 0$ for each $i$.
Other problem parameters include
$T=298$ K and Stefan--Maxwell diffusivities
\begin{align}
\mathscr{D} = 
 \begin{pmatrix}
& 21.87 & 16.63 & 23.15 \\
21.87 & & 16.40 & 22.85 \\ 
16.63 & 16.40 & & 16.02 \\ 
23.15 & 22.85 & 16.02 &  
 \end{pmatrix} \times 10^{-6} \text{ m}^2 \text{ s}^{-1},
\end{align}
and $M_1 = 0.028$ kg mol\textsuperscript{$-1$}, $M_2 = 0.032$ kg 
mol\textsuperscript{$-1$}, $M_3 =0.044$ kg mol\textsuperscript{$-1$}, $M_4 = 
0.018$ kg mol\textsuperscript{$-1$} respectively.

The original mesh used in \cite{van2022augmented} is shown in 
\Cref{fig:lungs_original}. It was made with Gmsh \cite{geuzaine2009} and has 
404174 cells, which is too fine to be the coarsest mesh in a geometric multigrid 
cycle if we want to consider several refinements and polynomial degrees. Hence, we 
use a simpler geometry with only 2407 cells consisting of 3 cylinders attached at 
120 degree angles in a two-dimensional plane using ngsPETSc 
\cite{betteridge2024ngspetsc, schoberl1997netgen}, as shown in 
\Cref{fig:lungs_simplified}. The cylinders crudely represent the trachea and the 
two primary bronchi.

The average linear and nonlinear iteration counts for the monolithic GMG 
preconditioner with AL smoother are displayed in \Cref{tab:lungs_its}. The 
preconditioner yields discretization-robust GMRES iteration counts on the 
simplified lung geometry, at least as far as can be solved on the computational
resources available.

We then solved the problem on the original mesh, refined twice uniformly with 
$k=1$ (i.e.~$\text{RT}_1-\text{DG}_0$ finite elements). In \Cref{fig:oxygen} the 
mole fraction of oxygen is displayed, and in \Cref{fig:lungs_zoom} the mole 
fraction and mass flux streamlines of carbon dioxide in one of the bronchioles 
are shown. This problem has approximately 316 million degrees of freedom and was 
solved in 6 Newton steps each requiring an average of 8.83 preconditioned GMRES 
iterates. The $L^2(\Omega)$-error of the mass-average constraint \eqref{eq:vbulk} and the 
volumetric equation of state \eqref{eq:state} are $6.3 \times 10^{-6}$ and $8.0 
\times 10^{-5}$ respectively. Without the preconditioner this high resolution 
simulation would not have been possible. The ratio of the largest cell diameter to 
the smallest cell diameter is 51, demonstrating that the preconditioner remains 
effective on mildly anisotropic meshes.

\begin{figure}
 \centering
 \begin{subfigure}{0.3\textwidth}
 \includegraphics[width=\linewidth]{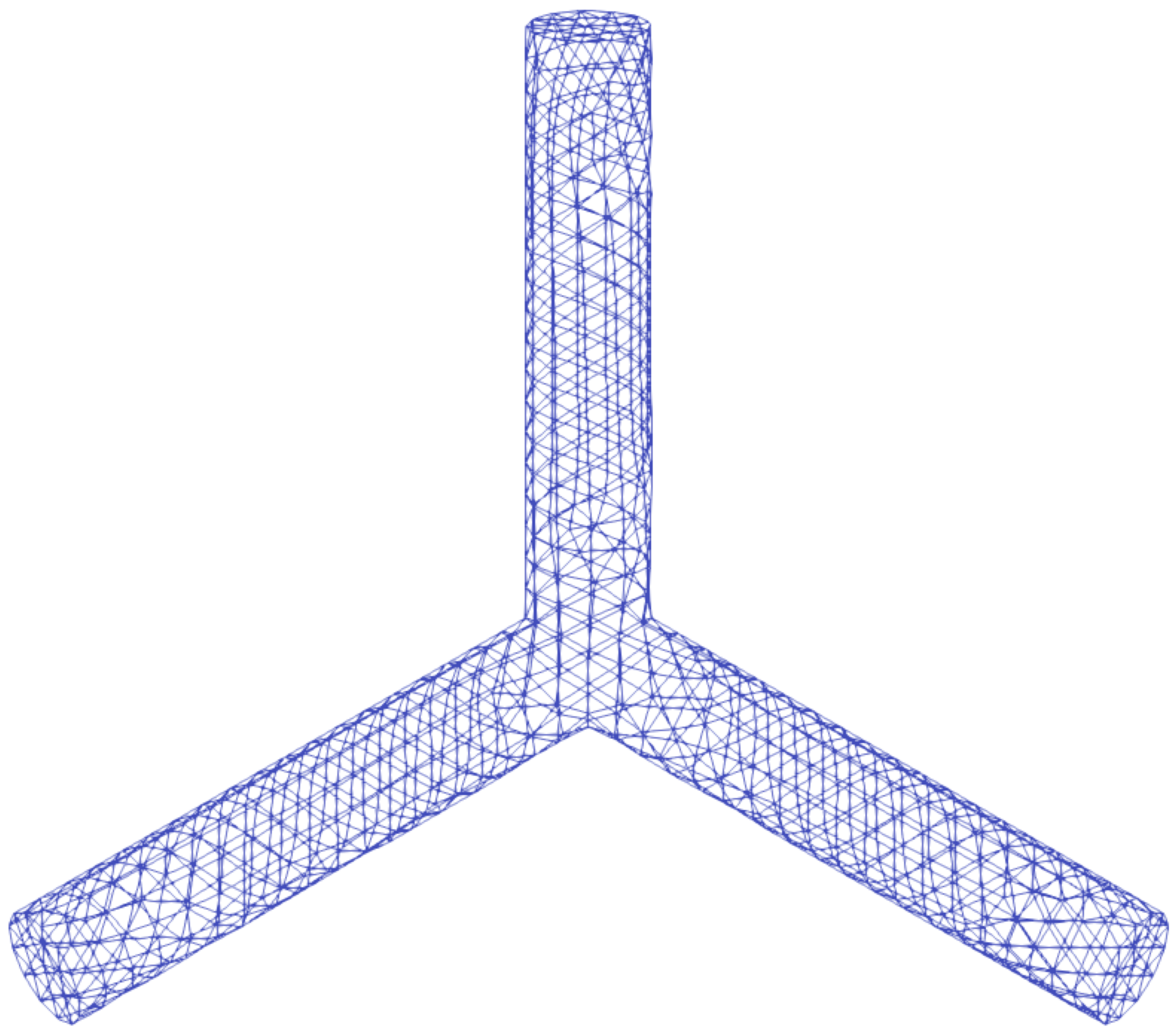}
 \caption{}
 \label{fig:lungs_simplified}
 \end{subfigure}
 \begin{subfigure}{0.25\textwidth}
 \includegraphics[width=\linewidth]{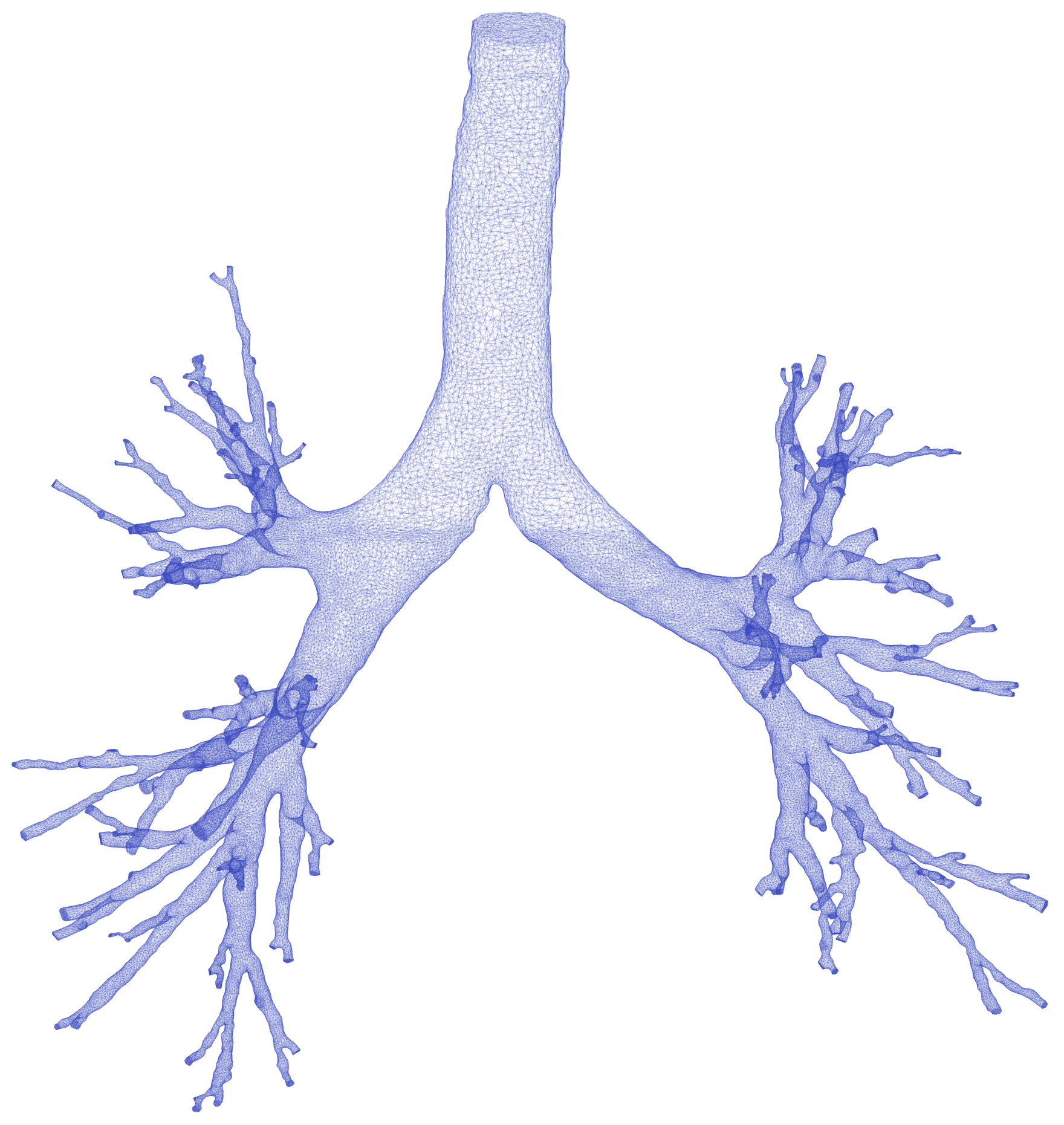}
 \caption{}
 \label{fig:lungs_original}
 \end{subfigure}
 \begin{subfigure}{0.28\textwidth}
 \includegraphics[width=\linewidth]{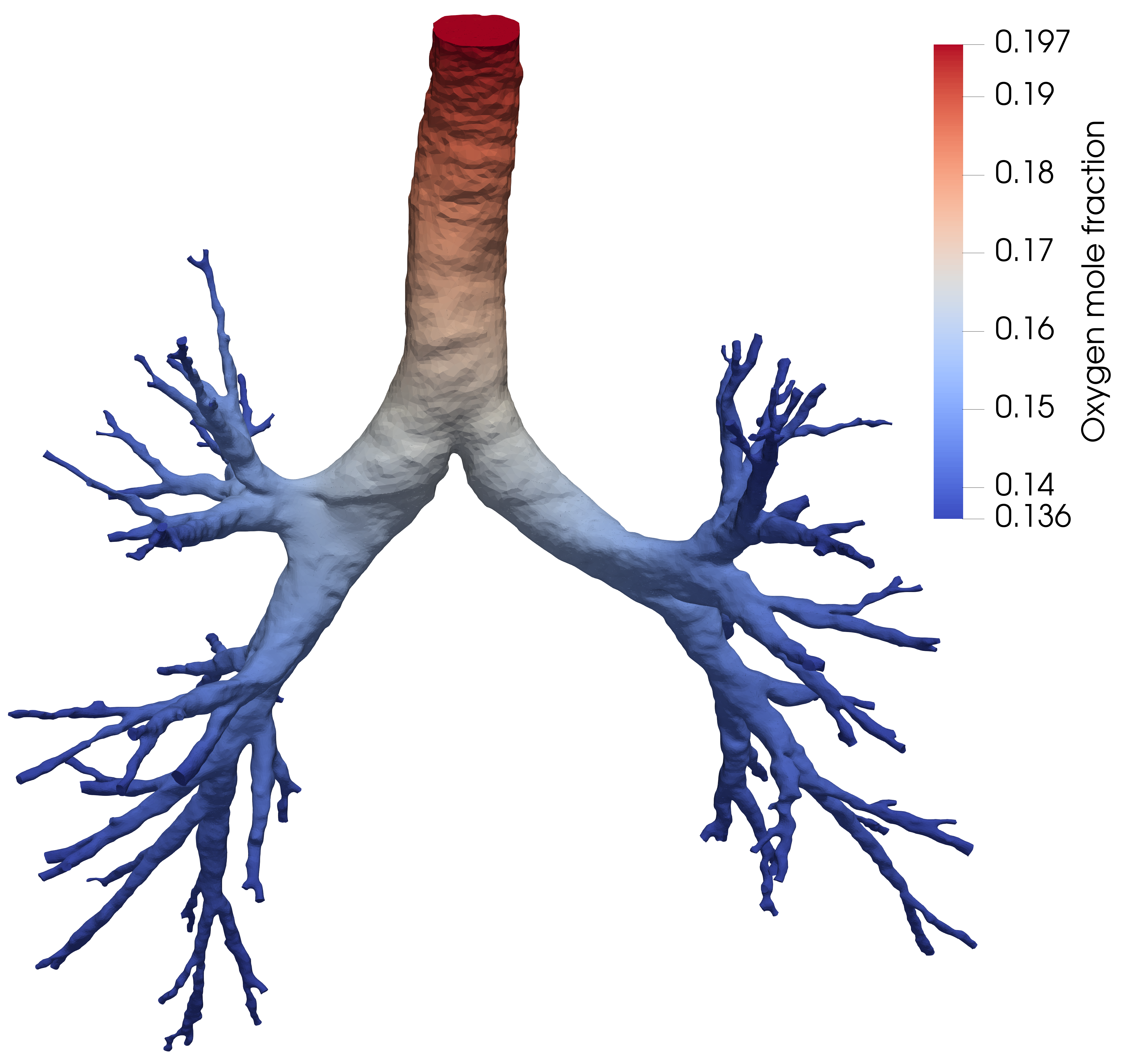}
 \caption{}
 \label{fig:oxygen}
 \end{subfigure}
 \caption{Simplified mesh, original mesh, and mole fraction of oxygen in the human 
 airways.}
 \label{fig:lungs}
\end{figure}

\begin{figure}
 \centering
 \includegraphics[width=0.5\linewidth]{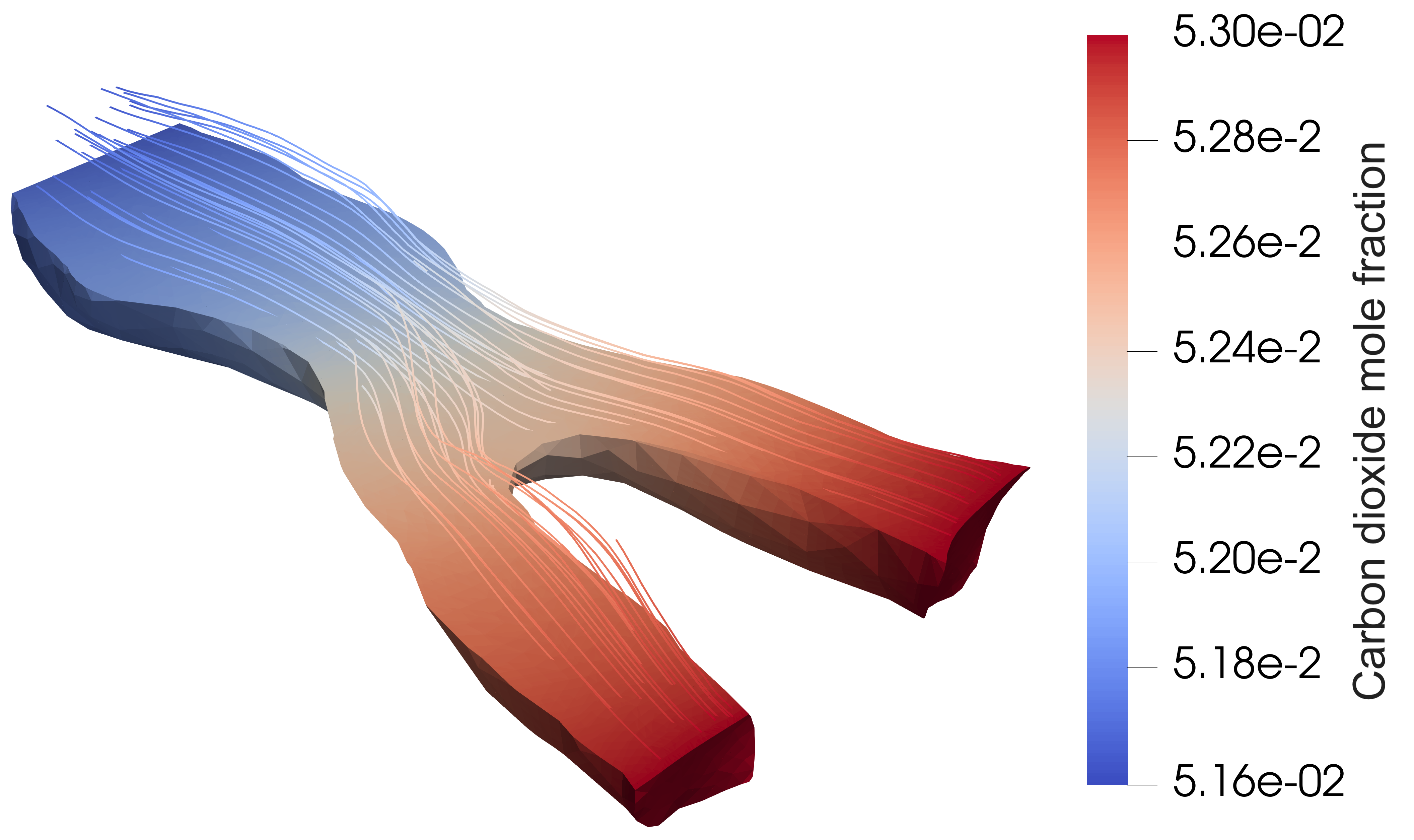}
 \caption{Carbon dioxide mole fraction on a cross section of a bronchiole and 
 streamlines of the carbon dioxide mass flux.
}
 \label{fig:lungs_zoom}
\end{figure}

\begin{table}[]
 \centering
 \scriptsize{
     \begin{tabular}{r|*{4}{c}}
     \toprule
     $m\setminus k$ 
     & 1 & 2 & 3 & 4
     \\
     \midrule
     1 & 6.25 (4) & 4.67 (3) & 4 (3) & 4 (3)  \\
     \rowcolor{gray!25}
     2 & 8.25 (4) & 4.67 (3) & 5 (4) & \\
     3 & 11.5 (4) & 6 (4) & & \\
     \rowcolor{gray!25}
     4 & 10.25 (4) & & & \\
     \bottomrule
     \end{tabular}
 }
 \caption{Average number of GMRES iterations for the diffusion in the human 
 airways problem of \Cref{sec:lungs} solved with the monolithic GMG preconditioner 
 with AL smoother. The number of Newton iterations is given in brackets.}
 \label{tab:lungs_its}
\end{table}

\section{Anisothermal mixtures}
\label{sec:anisothermal}
Thermal effects can incorporated in the OSM equations. Here, we consider a 
formulation of the anisothermal OSM problem that is structurally similar to the 
isothermal problem, because it incorporates 
the thermal energy as a pseudo-species 
\cite{van2022consolidated}. The analogues of velocity and chemical potential for 
thermal energy are the thermal velocity $v_0: \Omega \rightarrow \mathbb{R}^d$ in 
m s\textsuperscript{$-1$} and absolute temperature $T: \Omega \rightarrow 
\mathbb{R}$ in K respectively. In this setting the 0 index refers to thermal 
energy as a 
pseudo-species and the chemical species are indexed with $i>0$.

The diffusion driving forces of the chemical species are the same as in the 
isothermal setting, provided that one expands the species chemical potential 
gradients in terms of mole fraction gradients \cite{van2022consolidated}. 
To showcase the flexibility of our preconditioning strategy in handling different 
formulations, we therefore discretize the mole fractions instead of the 
chemical potentials.
This means that the concentrations can no longer be computed using
\cref{eq:ideal_constitutive_law}.
Instead, they must be computed by means of $c_i = c_T \chi_{i}$ where
$c_T$ is computed using the equation of state \cref{eq:state}.
To employ \cref{eq:state}, we must additionally solve for $p \in \mathbb{R}$.
Since $T$ is now unknown, the value of $p$ cannot be deduced from boundary
data \cref{eq:bc_mu} as in the isothermal case.
To match the number of unknowns with the number of equations, we must enforce an 
additional constraint for determining $p$. We choose the integral constraint
\begin{equation}
	N_T = \int_\Omega \frac{p}{RT} \; \mathrm{d}x,
	\label{eq:pressure_constraint}
\end{equation}
which can be obtained by integrating \eqref{eq:state}.
Note that $N_T = \int_\Omega 
c_T \; \mathrm{d}x$ is the total number of moles of all species present in 
$\Omega$. This is an experimentally measurable quantity and we assume that it is 
known.

The thermal driving force is 
proportional to the temperature gradient by a factor that generally depends on
temperature and composition.
For example, in the 
case of monatomic ideal gaseous mixtures, the anisothermal OSM equations become
\begin{align}
 \sum_{j=0}^n \hat{\bm{M}}_{ij} v_j = \begin{cases}
     -\frac{5}{2}c_TR\nabla T, & \text{if $i=0$}, \\
     -p\nabla \chi_i, & \text{if $i \in \{1,\dots, n\}$},
 \end{cases}
\end{align}
where $\hat{\bm{M}}$ is an $(n+1) \times (n+1)$ thermal Onsager transport matrix 
defined as
\begin{align}
 \hat{\bm{M}} =
 \begin{pmatrix} \hat{\bm{M}}_{00}& \hat{m}_{0}^{T}\\ \hat{m}_{0} & \bm{M} +  
 \frac{\hat{m}_{0}\hat{m}_{0}^{T}}{\hat{\bm{M}}_{00}}\end{pmatrix}.
\end{align}
The top left entry of $\hat{\bm{M}}$ is $\hat{\bm{M}}_{00}:= 
\frac{\rho^2\hat{C}_p^2T}{k_\text{mix}}$ where $\hat{C}_p:=\frac{5c_TR}{2\rho}$ is 
the specific constant pressure heat capacity in J K\textsuperscript{$-1$} 
kg\textsuperscript{$-1$}, and $k_\text{mix}$ is the thermal conductivity of the 
mixture in J K\textsuperscript{$-1$} m\textsuperscript{$-1$} 
s\textsuperscript{$-1$}. Furthermore, $\hat{m}_0^T = \begin{pmatrix}
 \hat{\bm{M}}_{01} & \dots & \hat{\bm{M}}_{0n}
\end{pmatrix}$ and its entries are given by
\begin{align}
 \hat{\bm{M}}_{0i} = - \frac{\hat{\bm{M}}_{00}}{\rho 
 \hat{C}_{p}T}\sum_{j=1}^{n}\bm{M}_{ij}\mathscr{D}_{j}^{T} \qquad &\forall i \in 
 \{1, \dots, n \}.
\end{align}
Here $\mathscr{D}_i^T$ is the Soret diffusivity of species $i$ given by
\begin{align}
 \mathscr{D}_{i}^{T}= \frac{1}{\rho}\sum_{j=1}^{n}\rho_{j}\mathcal{A}_{ij},
\end{align}
where $\mathcal{A}_{ij}$, the Newman--Soret diffusivity of species $i$ relative to 
species $j$ in m\textsuperscript{$d-1$} s\textsuperscript{$-1$}, is a given 
constant, and 
$\rho_j:= M_jc_j$ is the density of species $j$ in kg m\textsuperscript{$-d$}. We 
have the same mass continuity equations as in the isothermal OSM equations and a 
steady-state heat balance equation for monatomic ideal gaseous mixtures:
\begin{align}
 0 = \begin{cases}
     -\frac{5}{2}\nabla \cdot (pv_{0}), & \text{if $i=0$}, \\
     \nabla \cdot (c_{i}v_{i}), & \text{if $i \in \{1, \dots, n \}$}.
 \end{cases}
\end{align}

\subsection{Gas separation chamber}
\label{sec:separation}
This problem is adapted from \cite[Sec.~7.3]{van2022consolidated}. We consider a 
closed chamber containing an equimolar mixture of helium, argon, and krypton 
initially kept at room temperature and atmospheric pressure. The right wall is 
then heated with a fixed total heat input and the left side of the chamber is held 
adjacent to a cooling bath kept at room temperature. A schematic of the chamber is 
shown in \Cref{fig:gas_separation_chamber_schematic}. When the steady-state is 
reached, there is a temperature gradient across the chamber and the species 
arrange themselves accordingly. Helium, the lightest gas, will accumulate on the 
hot side of the chamber, and krypton, the heaviest gas, will accumulate on the 
cool side of the chamber at steady-state \cite{hafskjold1993molecular}. We seek to 
compute this steady-state.

The mixing rule for the thermal conductivity is given by \cite{udoetok2013thermal}
\begin{align}
 k_{\text{mix}} = 
 \frac{1}{2}\frac{k_{\text{He}}k_{\text{Ar}}k_{\text{Kr}}}{\chi_{\text{He}}k_{\text{Ar}}k_{\text{Kr}}+
  \chi_{\text{Ar}}k_{\text{He}}k_{\text{Kr}} + 
 \chi_{\text{Kr}}k_{\text{He}}k_{\text{Ar}}} + 
 \frac{1}{2}(k_{\text{He}}\chi_{\text{He}}+ 
 k_{\text{Ar}}\chi_{\text{Ar}}+k_{\text{Kr}}\chi_{\text{Kr}}),
\end{align}
where the thermal conductivities of helium, argon, and krypton are 
$k_\text{He}=1566$ J K\textsuperscript{$-1$} m\textsuperscript{$-1$} 
s\textsuperscript{$-1$}, $k_\text{Ar}=178$ J K\textsuperscript{$-1$} 
m\textsuperscript{$-1$} s\textsuperscript{$-1$}, and $k_\text{Kr}=95$ J 
K\textsuperscript{$-1$} m\textsuperscript{$-1$} s\textsuperscript{$-1$} 
respectively. The Newman--Soret diffusivity matrix $\mathcal{A}$ is
\begin{align}
 \mathcal{A} = \begin{pmatrix}
     0 & -0.2910 & -0.2906 \\
     0.2910 & 0 & 0.004 \\
     0.2906 & -0.004 & 0 
 \end{pmatrix} \times 10^{-4} \text{ m}^2 \text{ s}^{-1},
\end{align}
and the Stefan--Maxwell diffusivity matrix $\mathscr{D}$ is
\begin{align}
 \mathscr{D} = 
 \begin{pmatrix}
     & 0.7560 & 0.6653 \\
     0.7560 & & 0.2467 \\
     0.6653 & 0.2467 & \\
 \end{pmatrix} \times 10^{-4} \text{ m}^2 \text{ s}^{-1}.
\end{align}
The molar masses of helium, argon, and krypton are $0.004$ kg 
mol\textsuperscript{$-1$}, $0.03995$ kg mol\textsuperscript{$-1$}, and $0.08380$ 
kg mol\textsuperscript{$-1$} respectively. 
We set $v_\text{bulk}=0$ and $r_i = 0$ for all $i$.
One can deduce the value of the constraint data $N_T$ in 
\cref{eq:pressure_constraint} using the fact that the mixture
was initially at room temperature and atmospheric pressure.

We impose zero-flux boundary conditions for the thermal velocity along the 
insulated walls, i.e.~adiabatic walls, and zero-flux boundary conditions for the 
mass fluxes along all walls, i.e.~impermeable walls. On the heated wall we enforce 
$\frac{5}{2}pv_0 = 1\times10^{5}$ J m\textsuperscript{$-2$} 
s\textsuperscript{$-1$} which relates the thermal velocity to the total heat input 
along the heated wall. We enforce $T=300$ K along the cooled wall. Due to the full 
Neumann boundary conditions on the mass fluxes we need extra constraints for the 
system to admit a unique solution
\cite{baierreinio2024highorderfiniteelementmethods}.
We impose the integral constraints
\begin{align}
 N_i = \int_\Omega c_i \; \mathrm{d}x \qquad \forall i \in \{1,\dots, n \},
 \label{eq:concentration_constraint}
\end{align}
where $N_i$ is the total number of moles of species $i$,
chosen consistently with the requirement that $N_T = \sum_{i=1}^n N_i$.
In the next subsection we 
will discuss how to enforce integral constraints like 
\eqref{eq:pressure_constraint} and \eqref{eq:concentration_constraint}.

\begin{figure}
 \centering
 \begin{subfigure}{0.45\textwidth}
     \includegraphics[width=\linewidth]{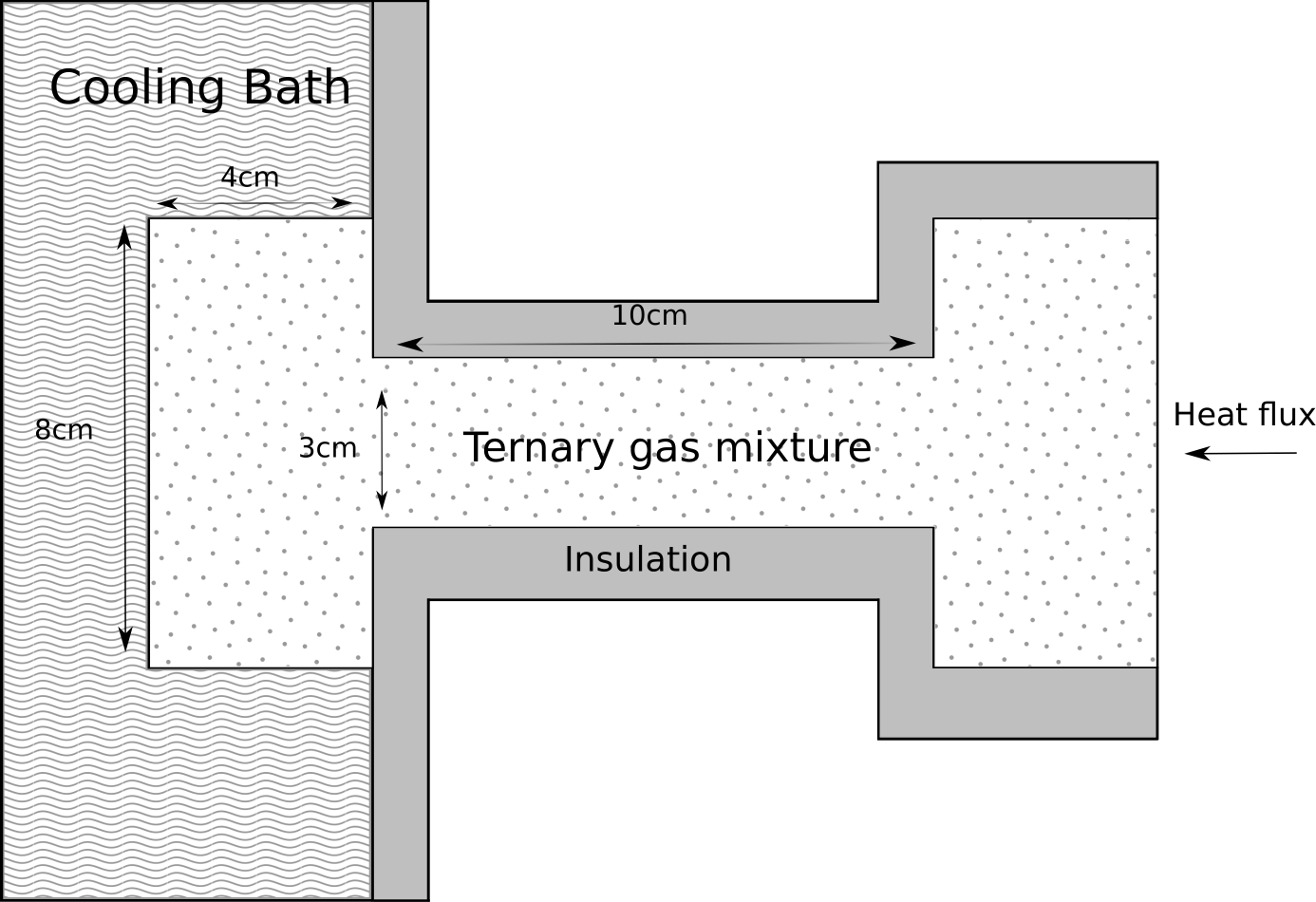}
     \caption{}
     \label{fig:gas_separation_chamber_schematic}
 \end{subfigure}
 \begin{subfigure}{0.5\textwidth}
     \includegraphics[width=\linewidth]{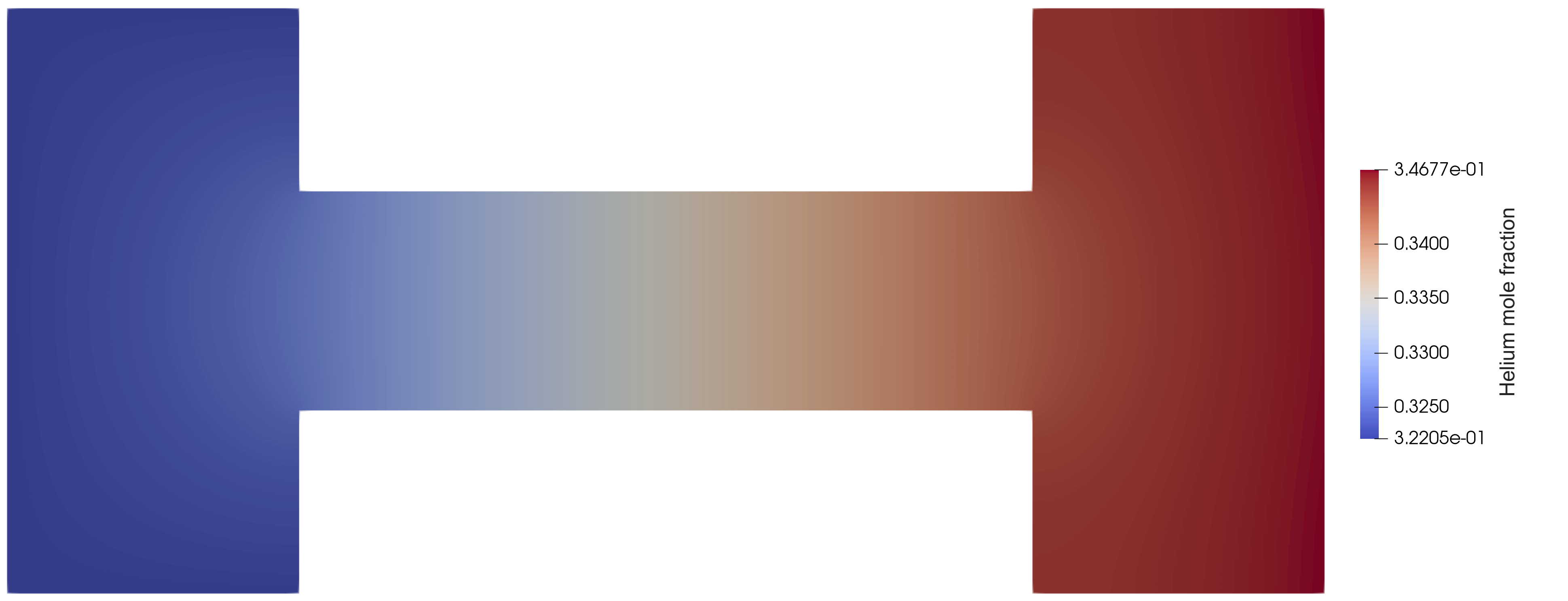}
     \vspace{1em}
     \caption{}
     \label{fig:helium}
 \end{subfigure}
 \caption{(a) Schematic of the gas separation chamber and (b) helium mole 
 fraction. The schematic in (a) is reproduced from \cite{van2022consolidated}.}
 \label{fig:gas_separation_chamber}
\end{figure}

\subsection{Integral constraints}
\label{sec:integral_constraints}

Due to the full Neumann boundary conditions on the mass fluxes, discretization
of the mass continuity equations $\nabla \cdot (c_{i}v_{i}) = 0$ yields a set
of algebraic equations in which $n$ of the equations are redundant
\cite{baierreinio2024highorderfiniteelementmethods}.
These redundant equations should be replaced with the constraints 
\cref{eq:concentration_constraint}. This is accomplished using the 
Woodbury formula \cite{hager1989updating} together with sparse direct solvers in
\cite{baierreinio2024highorderfiniteelementmethods}.
Since this paper considers regimes where sparse direct solvers are computationally 
unaffordable, we prefer the following alternative approach based on Lagrange multipliers.

We enforce the constraints \eqref{eq:concentration_constraint} by
introducing the decomposition
$\chi_i = \chi_i^\text{aux} + \lambda_i$ with
$\lambda_i \in \mathbb{R}$ an unknown scalar for each $i$.
We solve for the new variables
$\bar{\chi}^\text{aux}$ and $\bar{\lambda}$ instead of $\bar{\chi}$. Let $s_i \in 
\mathbb{R}$ be the test function corresponding to $\lambda_i$.
We enforce \eqref{eq:concentration_constraint} weakly using
\begin{align}
 \int_\Omega c_is_i \, \mathrm{d}x = \frac{1}{V} \int_\Omega N_is_i \, 
 \mathrm{d}x\qquad &\forall i \in \{1, \dots, n \},
 \label{eq:weak_concentration_constraint}
\end{align}
where $V$ denotes the volume of $\Omega$. We also fix $\chi_i^\text{aux}$ to be a 
fixed but arbitrary value at a mesh vertex, to
eliminate its indeterminacy up to constant additive shifts.

The total molar constraint \eqref{eq:pressure_constraint} can be 
enforced analogously to \eqref{eq:weak_concentration_constraint}, by
using a test function associated to the unknown scalar $p \in \mathbb{R}$.

\subsection{Preconditioners for anisothermal mixtures}
We extend the AL technique to include the thermal energy pseudo-species by adding 
the term 
$\kappa_0^l(\frac{5}{2}\nabla \cdot (pv_0))(\nabla \cdot \tau_0)$ to the weak 
formulation. It remains to deal with the real variables for the integral 
constraints in an appropriate manner. First of all, each integral constraint 
creates a dense row and column in the Jacobian. If the CSR format were used, each 
matrix row would be stored on a single processor, and the action of this dense row 
would require the whole system state to be loaded onto that processor. This would 
be extremely costly when running in parallel. Fortunately, assembly of these dense 
rows is not necessary as GMG is implemented matrix-free. Since the diagonal block 
corresponding to the integral constrains is a small $(n+1) \times (n+1)$ matrix, 
we can solve this block exactly (in exact arithmetic) in matrix-free fashion using 
$n+1$ GMRES iterations. As a result, the outer iteration changes to flexible GMRES.

Linear and nonlinear iteration counts of the monolithic GMG preconditioner with AL 
smoother for the gas separation chamber problem are displayed in 
\Cref{tab:separation}. We obtain mildly increasing FGMRES iteration counts as $m$ 
is increased but robustness with respect to $k$; the robustness with respect to 
$m$ improves as $k$ is increased. The mole fraction of helium is depicted in 
\Cref{fig:helium}. We note a separation of 7.5\% under a temperature difference of 
approximately 93 K.

\begin{table}[]
 \centering
 \scriptsize{
     \begin{tabular}{r|*{5}{c}}
     \toprule
     $m\setminus k$ 
     & 1 & 2 & 3 & 4 & 5
     \\
     \midrule
     1 & 7.25 (4) & 5.75 (4) & 5.25 (4) & 5 (4) & 4.5 (4) \\
     \rowcolor{gray!25}
     2 & 9.25 (4) & 6 (4) & 5.75 (4) & 6.25 (4) & 6.25 (4) \\
     3 & 12.75 (4) & 6.5 (4) & 6 (4) & 7 (4)  & 6.75 (4) \\
     \rowcolor{gray!25}
     4 & 17.75 (4) & 7.75 (4)& 7.75 (4) & 7.5 (4) & 6 (4)\\
     5 & 25 (4) & 9 (4)& 9 (4)& 7 (4)& 6.5 (4) \\
     \bottomrule
     \end{tabular}
 }
 \caption{Average number of FGMRES iterations for the gas separation chamber 
 problem of \Cref{sec:separation}, solved with the monolithic GMG preconditioner 
 with AL smoother. The number of Newton iterations is given in brackets.}
 \label{tab:separation}
\end{table}

\section{Nonideal mixtures}
\label{sec:nonideal}
The general setting permits nonideal mixing of nonideas gases and condensed phases.
In this case, the chemical potentials and total concentration are given by 
algebraic constitutive laws
\begin{align}
 \mu_i - \mu_i^\text{ref} &= G_i(T, p, \bar{\chi}) \qquad \forall i \in \{1, 
 \dots, n \}, \label{eq:nonideal_constlaw_G} \\
 1/c_T &= \sum_{j=1}^n \chi_j V_j(T,p, \bar{\chi}),
\end{align}
where $G_i$ and $V_i$ are known partial molar Gibbs functions and partial molar 
volume functions respectively. For ideal gaseous mixtures $G_i = RT\log \left ( 
\chi_i\frac{p}{p_\text{ref}} \right )$ and $V_i = RT/p$ for each $i\in \{1,\dots, 
n\}$. 
In the ideal gaseous setting
$\bar{\chi}$ can be expressed explicitly in terms of $\bar{\mu}$, so only 
one of these variables needs to be discretized. In
nonideal mixtures both $\bar{\chi}$ and $\bar{\mu}$ are discretized, and the 
$L^2(\Omega)$-projection of \eqref{eq:nonideal_constlaw_G} is enforced with 
$\chi_i$ sought in $\text{DG}_{k-1}$ for each $i \in \{1, \dots, n \}$
\cite{baierreinio2024highorderfiniteelementmethods}.

Furthermore, we can include pressure and convective effects by precomputing a 
consistent total mass flux $\mathcal{J} := \rho v_{\text{bulk}}$ and pressure
using the Stokes equations
\begin{align}
 - \frac{1}{\rho^\text{ref}}\nabla \cdot \mathcal{C}^{-1}\epsilon (\mathcal{J}) + 
 \nabla p &= \rho F, \\
 \nabla \cdot \mathcal{J} & = \sum_{i=1}^n M_ir_i,
\end{align}
where $F: \Omega \rightarrow \mathbb{R}^d$ is the prescribed body force in kg m s\textsuperscript{$-2$} and $\mathcal{C}: \mathbb{R}^{d\times d}_\text{sym} 
\rightarrow \mathbb{R}^{d\times d}_\text{sym}$ is the compliance tensor 
corresponding to the constitutive law, which we take to be Newtonian. Here we have 
used the approximation $v_{\text{bulk}}\approx \rho 
v_{\text{bulk}} / \rho^\text{ref}$ where $\rho^\text{ref}$ is a constant reference 
density. This is a reasonable approximation for nearly incompressible fluids. By 
considering $\mathcal{J}$ instead of $v_\text{bulk}$, boundary conditions that are 
consistent with the mass-average constraint \cref{eq:vbulk} can be imposed 
without having to couple the Stokes and OSM problems.
Precomputing a total mass 
flux and pressure using the Stokes equations allows us to drop the isobaric 
assumption in the OSM equations, so that in the isothermal setting the diffusion 
driving forces become
\begin{align}
 d_i = -c_i\nabla \mu_i + \frac{M_ic_i}{\rho}\nabla p \qquad \forall i \in \{1, 
 \dots, n \}.
\end{align}
Similarly to \cite{baierreinio2024highorderfiniteelementmethods} this requires the 
introduction of the inverse density $\psi \coloneqq {\rho}^{-1}$ as a variable and 
the enforcement of the $L^2(\Omega)$-projection of $1/\psi = \sum_{i=1}^nM_ic_i$. 
In the discrete problem, we use a Taylor--Hood $\text{CG}_k$--$\text{CG}_{k-1}$ 
finite element pair for $\mathcal{J}$ and $p$, and $\text{CG}_{k-1}$ for $\psi$.

\subsection{Benzene-cyclohexane mixture}
Consider a microfluidic device in which pure benzene and cyclohexane are piped in 
through two separate inlets and mix before flowing through the outlet. This mixing 
is nonideal, as the benzene-cyclohexane interactions are weaker than 
benzene-benzene interactions. This problem is an adaptation of 
\cite[Sec.~5.2.]{baierreinio2024highorderfiniteelementmethods} in which the full 
coupling of the OSM equations with the Stokes equations is considered.

The mixing of these species is captured by a Margules-type thermodynamic 
constitutive relation \cite{perry1950chemical}
\begin{subequations} \label{eq:Margules}
 \begin{align}
     \mu_1 &= \frac{p}{c_1^{\text{ref}}} + RT\log (\chi_1) + RT \chi_2^2(A_{12} + 
     2(A_{21} - A_{12})\chi_1), \label{eq:Margules_a} \\
     \mu_2 &= \frac{p}{c_2^{\text{ref}}} + RT\log (\chi_2) + RT \chi_1^2(A_{21} + 
     2(A_{12} - A_{21})\chi_2), \label{eq:Margules_b} \\
     \frac{1}{c_T} &= \frac{\chi_1}{c_1^\text{ref}} + 
     \frac{\chi_2}{c_2^\text{ref}}, \label{eq:Margules_c}
 \end{align}
\end{subequations}
where $A_{12}=0.4498$ and $A_{21}=0.4952$. For each $i \in \{1,2\}$ we 
prescribe no-normal fluxes $J_i \cdot \hat{n} = 0$ on all walls, and on the inlet 
and outlet $J_i$ has zero tangential component and a parabolic profile for the 
normal component. The magnitude of the parabolic profiles are 
$M_1c_1^{\text{ref}}v_1^{\text{ref}}$ and $M_2c_2^{\text{ref}}v_2^{\text{ref}}$ 
for $J_1$ and $J_2$ respectively. For the problem to admit a unique solution we 
enforce integral constraints $\int_\Omega \chi_1 + \chi_2 -1 \; \mathrm{d}x=0$ and 
$\int_{\partial \Omega_\text{out}}M_1c_1 - M_2c_2\; \mathrm{d}x=0$. Further, we 
define the parameters $\mathscr{D}_{12}=\mathscr{D}_{21}=2.1 \times 10^{-9} \; 
\text{m}^2 \; \text{s}^{-1}$, $M_1 = 0.078 \; \text{kg} \; \text{mol}^{-1}$, $M_2 
= 0.084 \; \text{kg} \; \text{mol}^{-1}$, $M_1c_1^{\text{ref}} = 876 \; \text{kg} 
\; \text{m}^{-3}$, $M_2c_2^{\text{ref}} = 773 \; \text{kg} \; \text{m}^{-3}$, 
$v_1^{\text{ref}} = 10 \times 10^{-6} \; \text{m} \; \text{s}^{-1}$, and 
$v_2^{\text{ref}} = \frac{c_1^{\text{ref}}}{c_2^{\text{ref}}} v_1^{\text{ref}}$. 

To precompute a consistent $\mathcal{J}$ and $p$ using the Stokes problem we 
assume $\mathcal{J}=0$ along the walls, and on the inlet and outlet the bulk 
velocity has zero tangential component and a parabolic profile for the normal 
component. The magnitude of the parabolic profile on the benzene and cyclohexane 
inlets are $M_1c_1^{\text{ref}}v_1^{\text{ref}}$ and $M_2c_2v_2^{\text{ref}}$ 
respectively, and the magnitude of the parabolic profile on the outlet is 
$M_1c_1^{\text{ref}}v_1^{\text{ref}} + M_2c_2v_2^{\text{ref}}$ which ensures a 
mass balance if the inlet and outlet are the same size. The shear and bulk 
viscosity are $\eta = 6 \times 10^{-4} \; \text{Pa} \; \text{s}$ and $\zeta = 
10^{-7} \; \text{Pa} \; \text{s}$.

\subsection{Preconditioners for nonideal mixtures}
We apply the AL preconditioning strategy as before, so we only need to incorporate 
the diagonal blocks in the Jacobian corresponding to the extra variables in this 
application. The $2 \times 2$ diagonal block in the Jacobian corresponding to 
$\bar{\chi}$ is solved exactly as the $\text{DG}$ spaces decouple across cells. 
The diagonal block corresponding to $\psi$ is a $\text{CG}$ mass matrix in the 
Newton linearization, and smoothed scalably using Chebyshev--Jacobi iteration 
\cite{wathen2008chebyshev}. The real variables corresponding to the integral 
constraints are treated as in \Cref{sec:integral_constraints}. The domain geometry 
is created using ngsPETSc \cite{betteridge2024ngspetsc, schoberl1997netgen}. The 
coarse mesh has 5478 cells in total. We employ continuation on the magnitude of 
the boundary data to aid convergence of the nonlinear solver.

In \Cref{tab:benzene-cyclohexane} the average number of FGMRES and Newton 
iterations are displayed and we observe mildly increasing FGMRES iteration counts 
with $m$, but robustness with $k$. The mole fraction of benzene and streamlines 
of the benzene mass flux are plotted in \Cref{fig:benzene}.

\begin{figure}
 \centering
 \includegraphics[width=0.5\linewidth]{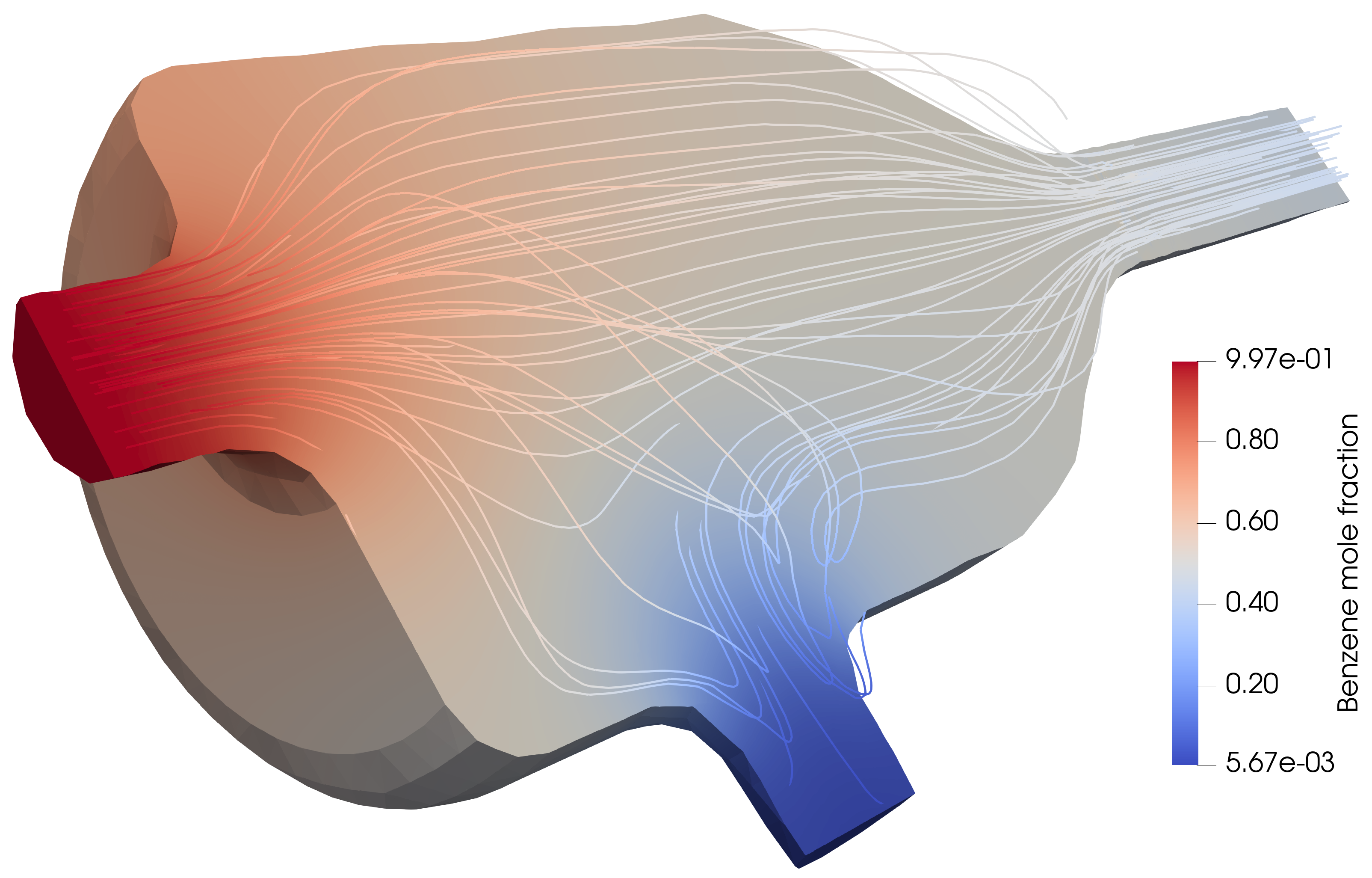}
 \caption{Benzene mole fraction on a cross section of the microfluidic 
 device and streamlines of the benzene mass flux.}
 \label{fig:benzene}
\end{figure}

\begin{table}[]
 \centering
 \scriptsize{
     \begin{tabular}{r|*{3}{c}}
     \toprule
     $m\setminus k$ 
     & 2 & 3 & 4
     \\
     \midrule
     1 & 9.88 (4.8) & 7.25 (4.8) & 6.54 (4.8)  \\
     \rowcolor{gray!25}
     2 & 12.21 (4.8) & 8.92 (4.8) &  \\
     \bottomrule
     \end{tabular}
 }
 \caption{Average number of FGMRES iterations for the benzene-cyclohexane mixture 
 of \Cref{sec:nonideal} solved with the monolithic GMG preconditioner with AL 
 smoother. The average number of Newton iterations per continuation step is given in brackets.}
 \label{tab:benzene-cyclohexane}
\end{table}

\section{Electrolytic mixtures}
\label{sec:electrolytic}

Electrolytic solutions can be modeled by incorporating local electroneutrality in 
the OSM equations. Local electroneutrality poses a constraint on the charged 
species that local excess charge density is zero everywhere, i.e.
\begin{align}
 \sum_{j=1}^n z_jc_j = 0,
\end{align}
where $z_j$ is the equivalent charge of species $j$. This constraint makes the 
species concentrations linearly dependent, and thus one species can be eliminated 
from the equations. By transforming the OSM equations using a \textit{salt-charge 
basis}, $n_c-1$ neutral salts are modeled instead of $n_c$ charged species 
\cite{van2023structural}. This construction gives rise to a natural definition of 
electricity described by an electric potential $\Phi$ in V and a current density $i$ 
in A m\textsuperscript{$-d$}. Analogously to thermal energy in the gas separation 
chamber, electricity will appear as a pseudo-species in the electroneutral OSM 
equations.

The electroneutral OSM equations structurally look identical to 
\cref{eq:augmented_OSM_} except $\widetilde{\bm{M}}^\gamma$ is replaced by 
congruent 
$\tilde{\bm{M}}^\gamma_{\bm{Z}}$ after transformation using the salt-charge basis; 
see \cite{baier2025finite, van2023structural} for details. Often problem parameters 
and constitutive laws in the chemical potential formulation are not 
experimentally available, and it is therefore imperative to formulate the 
electroneutral OSM equations in terms of mole fractions $\bar{\chi}$. 
When employing the salt-charge transformation it is also more convenient to use 
mole fluxes $N_j := c_j v_j$ instead of mass fluxes $J_j = M_j N_j$.

\subsection{Hull cell}
The Hull cell is an experimental device used to evaluate electroplating behavior 
in an electrolytic solution. The cell consists out of an electrolytic mixture 
inside a trapezoidal container with electrodes placed on the nonparallel sides. 
Under an applied electric current plating and stripping of cations occurs on the 
electrodes. We consider a nonideal binary electrolytic solution used in 
lithium-ion batteries: lithium hexafluorophosphate (LiPF\textsubscript{6}) 
dissolved in ethyl methyl carbonate (EMC) and ethylene carbonate (EC). In contrast 
to previous work \cite{baier2025finite}, we solve for the logarithms of the mole 
fractions $\log (\bar{\chi})$ as this makes integration by parts directly feasible 
without the introduction of continuous counterparts of the mole fractions.

On the electrodes we enforce linearized Butler--Volmer boundary conditions, 
i.e.\ Robin boundary conditions relating $\Phi$ and $i$, as well as 
$(N_{\text{LiPF}_6} - 0.5i)\cdot \hat{n}=0$. Finally, integral constraints are 
incorporated to enforce the spatially averaged LiPF\textsubscript{6} and EMC mole 
fractions, and unity of the mole fraction sum.

\subsection{Preconditioners for electrolytic mixtures}
Since $(N_{\text{LiPF}_6} - 0.5i)\cdot \hat{n}=0$ is not a Dirichlet boundary 
condition and cannot naturally be enforced through the weak formulation either, 
its $L^2(\partial \Omega_\text{elec})$-projection is enforced instead. This 
overwrites the equations for the degrees of freedom on the boundary, and is not
compatible with our implementation of the AL preconditioner.
Thus, for this problem we have used the monolithic GMG preconditioner with vertex 
Vanka smoother instead. In \Cref{tab:hull-cell} the average number of FGMRES and 
Newton iterations are displayed and good robustness is demonstrated for all mesh 
refinements and $k \geq 2$.

\begin{figure}
 \centering
 \includegraphics[width=0.7\linewidth]{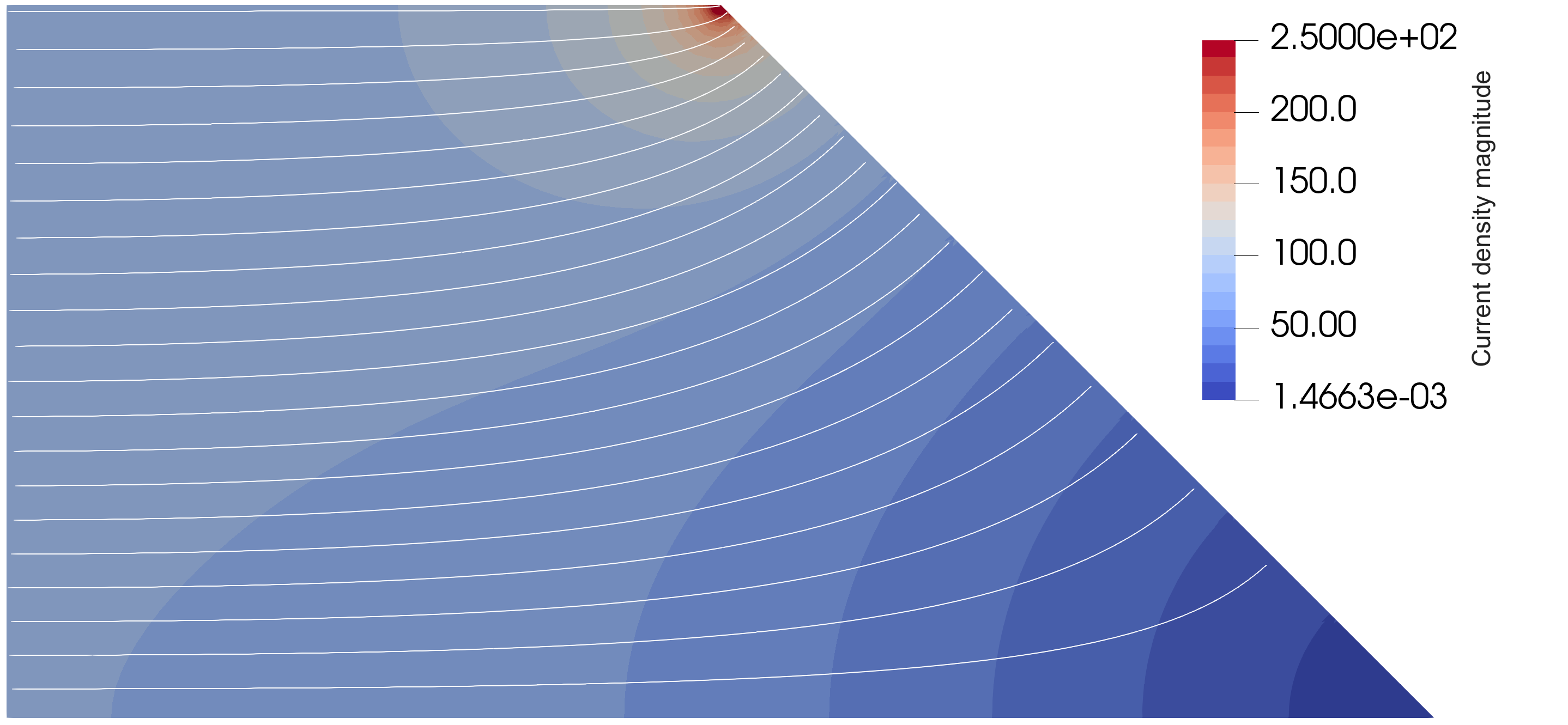}
 \caption{Magnitude of the current density in the steady-state Hull cell. Current 
 flows along the streamlines from the cathode on the left to the anode on the 
 right.}
 \label{fig:hull-cell}
\end{figure}

\begin{table}[]
 \centering
 \scriptsize{
     \begin{tabular}{r|*{5}{c}}
     \toprule
     $m\setminus k$ 
     & 1 & 2 & 3 & 4 & 5
     \\
     \midrule
     1 & 6.6 (5) & 5.2 (5)& 4.2 (5)& 4.8 (5) & 4.2 (5)\\
     \rowcolor{gray!25}
     2 & 7.4 (5)& 5.6 (5)& 4.4 (5) & 5.0 (5)& 4.6 (5)\\
     3 & 8.4 (5)& 5.8 (5)& 4.6 (5)& 4.6 (5)& 4.4 (5)\\
     \bottomrule
     \end{tabular}
 }
 \caption{Average number of FGMRES iterations for the Hull cell problem solved 
 with the monolithic GMG preconditioner with vertex Vanka smoother and number of 
 Newton iterations in brackets.}
 \label{tab:hull-cell}
\end{table}

\section{Conclusions}
\label{sec:conclusions}
In this work we have introduced scalable preconditioners for a wide variety of 
multicomponent diffusion problems modeled by the OSM equations. In the isobaric, 
isothermal, ideal gaseous regime, discretization-robustness of an augmented 
Lagrangian preconditioner for a Picard linearization is proven and demonstrated 
numerically. This augmented Lagrangian preconditioner inspires a monolithic 
multigrid preconditioner with augmented Lagrangian smoother for the Newton 
linearization which achieves discretization-robust or mildly increasing GMRES 
iteration counts, scalability through the use of patch smoothers, and 
computational advantage compared to other preconditioners. The flexibility of the 
monolithic multigrid strategy is successfully demonstrated for mixtures with 
cross-diffusion, nonideal mixing, thermal, pressure, convective, and 
electrochemical effects.

\section*{Acknowledgments}
The authors are grateful to Pablo Brubeck and Jongho Park for their useful 
discussions and feedback about this work.

\bibliographystyle{siamplain}
\bibliography{references}
\end{document}